\documentclass[12pt]{article}
\usepackage{amssymb,amsmath,amsthm, amsfonts}
\usepackage{mathrsfs}
\usepackage{graphicx}
\usepackage{epsfig}
\usepackage{tikz}
\usepackage{mathtools}

\theoremstyle{plain}
\newtheorem{theorem}{Theorem}[section]
\newtheorem{lemma}{Lemma}[section]

\theoremstyle{definition}
\newtheorem{definition}{Definition}[section]

\newtheorem{remark}{Remark}[section]

\numberwithin{equation}{section}

\begin{document}
\allowdisplaybreaks
\title{\large\bf Asymptotic behavior of a quasilinear  Keller--Segel system
with  signal-suppressed motility}

\author{
{\rm Chi Xu$^{1}$, Yifu Wang$^{1,*}$
}\\[0.2cm]
{\it\small \rm School of Mathematics and Statistics,  Beijing Institute of Technology}\\
{\it\small \rm Beijing 100081, P.R. China$^1$}\\
}
\date{}

\maketitle
\renewcommand{\thefootnote}{\fnsymbol{footnote}}
\setcounter{footnote}{-1}
\footnote{$^{*}$Corresponding author. E-mail addresses:wangyifu@bit.edu.cn (Y.Wang), XuChi1993@126.com (C. Xu)}

\maketitle

\begin{abstract}
This paper is concerned with the density-suppressed motility model:
$u_{t}=\Delta (\displaystyle\frac{u^m}{v^\alpha}) +\beta uf(w),
v_{t}=D\Delta v-v+u,
w_{t}=\Delta w-uf(w)$
in a smoothly bounded  convex domain $\Omega\subset {\mathbb{R}}^2$, where $m>1$,  $\alpha>0, \beta>0$ and $D>0$ are parameters, the response function $f$ satisfies $f\in C^1([0,\infty)), f(0)=0, f(w)>0$ in $(0,\infty)$.  This system  describes the density-suppressed motility of  Eeshcrichia
coli cells in process of  spatio-temporal pattern  formation via  so-called
self-trapping mechanisms.   Based on the duality argument, it is shown that for suitable large $D$ the problem admits at least one
global weak solution $(u,v,w)$ which will asymptotically converge to the spatially uniform equilibrium   $(\overline{u_0}+\beta \overline{w_0},\overline{u_0}+\beta \overline{w_0},0)$ with $\overline{u_0}=\frac1{|\Omega|}\int_{\Omega}u(x,0)dx $
and $\overline{w_0}=\frac1{|\Omega|}\int_{\Omega}w(x,0)dx $  in $L^\infty(\Omega)$.
\end{abstract}


\vskip4mm

\section{Introduction}
Chemotaxis, a kind of oriented motion of cells and organisms in response to certain  chemicals in the environment, plays an outstanding role in the life of many cells and microorganisms, such as the transport of embryonic cells to developing tissues and immune cells to infection sites (\cite{Is, Murray}). The celebrated mathematical model
describing chemotactic migration processes at population level is Keller--Segel-type  system of the form
\begin{equation}\label{1.1}
\left\{
\begin{array}{ll}
u_{t}=\nabla\cdot(\gamma(u,v)\nabla u-u\phi(u,v)\nabla v),\quad &x\in\Omega,t>0,\\
v_{t}=d\Delta v-v+u,&x\in\Omega,t>0,
\end{array}
\right.	
\end{equation}
in a bounded domain $\Omega\subset \mathbb{R}^n$ where $u = u(x, t)$ denotes  the population density  and $v=v(x,t)$ is the concentration of  chemical substance secreted by the population itself (\cite{Keller70}).
The most striking features of \eqref{1.1}
is the ability of the constitutive ingredient  cross-diffusion thereof to describe the collective behavior of  cell populations
mediated by a chemoattractant. Indeed, a rich literature has revealed that
the Neumann initial-boundary value problem for the classical Keller--Segel system
%the spontaneous formation of singularities occurs in its classical version
%minimal
\begin{equation}\label{1.2}
\left\{
\begin{array}{ll}
u_{t}=\triangle  u- \nabla\cdot (u\nabla v),\quad &x\in\Omega,t>0,\\
v_{t}=d\Delta v-v+u,&x\in\Omega,t>0
\end{array}
\right.	
\end{equation}
 possesses  solutions blowing up in finite time with respect to the spatial $L^\infty$ norm of $u$
%obtained on letting $D \equiv \phi\equiv 1$
in two- and even higher-dimensional frameworks  under
some condition on the mass and the moment of the initial data (\cite{Hv, Herrero,JMPA}, see also the surveys \cite{Bellomo}).
Apart from that, when $\phi$ and $\gamma$ in  \eqref{1.1} are only smooth positive functions of
$u$ on $[0, \infty)$, a considerable literature
% its (1.1), namely the spontaneous formation of singularities known to occur in the classical (aka minimal)
%Keller-Segel system obtained on letting  in two- and higher-dimensional frameworks (cf. [9],
%[36] and also [20]), a large variety of results in the literature
underlines the crucial role of %
asymptotic beahvior of
the ratio $\frac {\gamma (u)}{\phi(u)} $ at large values of $u$
% the interplay
%between the chemotactic sensitivity $\phi$ and the diffusion rate $\gamma$
 with regard to the
occurrence of singularity phenomena (see recent progress in \cite{Ishida,WNon, WJDE2019}).

%In comparison to the latter case where ,
As a simplification  of  \eqref{1.1}, Keller--Segel system with density dependent motility
\begin{equation}\label{1.3}
\left\{
\begin{array}{ll}
u_{t}=\nabla\cdot(\gamma(v)\nabla u-u\phi(v)\nabla v),\quad &x\in\Omega,t>0,\\
v_{t}=d\Delta v-v+u,&x\in\Omega,t>0
\end{array}
\right.	
\end{equation}
 was proposed to describe the aggregation phase of Dictyostelium discoideum (Dd) cells in response to
the chemical signal cyclic adenosine monophosphate (cAMP) secreted by Dd cells in \cite{Keller70b}.
Here the signal-dependent diffusivity $\gamma(v)$ and chemotactic sensitivity function $\phi(v)$ are linked through
$$
\phi(v) = (\alpha-1)\gamma'(v),
$$
where  $\alpha\geq 0$ denotes the ratio of effective body length (i.e. distance between the signal-receptors) to the
walk length (see [19] for details). Notice that when  $\alpha=0$, there is only
one receptor in a cell and hence chemotaxis is driven  by the undirect effect of chemicals
  %due to %a local sensing mechanism without measuring
in  the absence of the chemical gradient sensing. In this case, \eqref{1.3} reads as
  \begin{equation}\label{1.4}
\left\{
\begin{array}{ll}
u_{t}=\triangle(\gamma(v)u),\quad &x\in\Omega,t>0,\\
v_{t}=d\Delta v-v+u,&x\in\Omega,t>0,
\end{array}
\right.	
\end{equation}
where  the considered diffusion process of the population is essentially Brownian,
 %$\gamma$ denotes the signal-dependent motility,
 assumption $\gamma'(v)<0$ accounts for the repressive effect
 of the chemical concentration on the  population motility (\cite{FTLH}). In the context of acyl-homoserine lactone (AHL)
 density-dependent motility, the extended  model of  \eqref{1.4} %three-component reaction-diffusion system of the form
\begin{equation}\label{1.5}
\left\{
\begin{array}{ll}
u_{t}=\Delta (u\gamma (v)) +\beta \displaystyle\frac {uw^2}{w^2+\lambda},&x\in \Omega,~t>0, \\
v_{t}=D\Delta v+u-v, &x\in\Omega,~t>0,\\
w_{t}=\Delta w-\displaystyle\frac {uw^2}{w^2+\lambda},&x\in\Omega,~t>0
\end{array}
\right. 	
\end{equation}
%with $f(w)=\frac {w^2}{w^2+\lambda}$
was proposed  in \cite{Liu} to advocate that spatio-temporal pattern of  Eeshcrichia
coli cells can be induced via  so-called
``self-trapping'' mechanisms, that is at low
AHL level, the bacteria undergo run-and-tumble random motion, while at high
AHL levels, the bacteria tumble incessantly and become immotile at the macroscale.

In comparison with  plenty of results on the  Keller--Segel system where the diffusion depends on the density of cells,
 the respective knowledge seems to be
 much less complete when the  cell dispersal explicitly depends on the chemical concentration via the motility function  $\gamma(v)$,
 which is due to considerable challenges  to the analysis  caused by the degeneracy of $\gamma(v)$ %may approach to zero
  as $v\rightarrow \infty$ from the mathematical point of view.
  %and thus  brings.
Indeed, to the authors's knowledge,  Yoon and Kim
  (\cite{YKim}) showed that in the case of $\gamma(v)=\frac{c_0}{v^k}$ for small $c_0$, problem \eqref{1.4}
 admits a global classical solutions in any dimensions.
 The smallness condition on $c_0$ is removed lately in \cite{AknYoon} for the parabolic-elliptic version of \eqref{1.4} with $0 <k < \frac n{(n-2)_+} $.
 Furthermore, for the full parabolic system \eqref{1.4} in the  three-dimensional setting, Tao and Winkler (\cite{TWM3AS}) showed the existence of certain globally weak solutions, which become eventually smooth and bounded for suitably small
initial data $u_0$ under the assumption
%$$ \gamma(v)\in C^3([0,\infty)), ~~\hbox{and there exist}~~\gamma_1,\gamma_2, \eta>0 ~\hbox{such that}~~
 %\gamma_1\leq \gamma(v)\leq\gamma_2, |\gamma'(v)|\leq \eta
 %~~\hbox{for all}~~ \geq 0
 %$$
\begin{eqnarray*}\label{H}
 &&(H)~~\gamma(v)\in C^3([0,\infty)),~\textnormal{and there exist $\gamma_1,~\gamma_2,\eta>0$
  such that $0<\gamma_1\leq \gamma(v)\leq \gamma_2,$}\nonumber \\
&&\textnormal{$|\gamma^{\prime}(v)|<\eta$ for all $v\geq 0$.}	
 \end{eqnarray*}
 %with positive constants $\gamma_1,\gamma_2$ and $\eta$.
  It should be remarked that based on
 the comparison method,  Fujie and Jiang (\cite{Fujiangcvpde}) obtained the
  uniform-in-time boundedness %of solutions %existence of globally classical solution
  to %the parabolic-elliptic version of
  \eqref{1.4} % always exists globally
 in two-dimensional setting for the more general motility function $\gamma$, and in the three-dimensional case  under a stronger growth condition on $1/\gamma$ respectively. In addition, they investigated  the asymptotic behavior  to the parabolic-elliptic analogue of
  \eqref{1.4}   under the assumption $\displaystyle\max_{0\leq v< +\infty}\frac {|\gamma'(v)|^2}{\gamma(v)}<+\infty$ or
  $\gamma(v)=v^{-k}$ with $0<k<\frac{n}{(n-2)_+}$ in \cite{Fujiangjde,Jiang}.

On the considered time scales of cell migration, %If the dispersal of chemotactic populations  phenomena %are modelled on not only small time scales,
e.g. metastatic cells moving in semi-solid medium, often it is relevant to take into account
the growth of the population.
 A prototypical
choice to accomplish this is the addition of logistic growth terms $\kappa u-\mu u^2$ in the cell equation \cite{Murray}.
From the mathematical point of view, the dissipative action %damping effect
of logistic-like growth possibly  prevents the occurrence of singularity %%ctic collapse
phenomena in various chemotaxis models.  For instance, for the chemotaxis-growth system (\cite{FTLH})
\begin{equation}\label{1.6}
\left\{
\begin{array}{ll}
u_{t}=\triangle(\gamma(v)u)+\mu u(1-u),\quad &x\in\Omega,t>0,\\
v_{t}=d\Delta v-v+u,&x\in\Omega,t>0,
\end{array}
\right.	
\end{equation}
it is shown in \cite{JKWSiam} that in two dimensional setting, the system
admits a unique global classical solution when the motility function
$ \gamma\in   C^3([0,\infty)), \gamma(v) > 0 $ and $\gamma'(v) < 0$  for all  $v\geq 0$, $\lim_{v\rightarrow \infty} \gamma(v)=0$ and
$\displaystyle\lim_{v\rightarrow \infty}\frac{\gamma ^{\prime}(v)}{\gamma(v)}$ exists, and even  the constant steady state $(1,1)$
is globally asymptotically stable if $\mu >\frac 1{16}\displaystyle\max_{0\leq v< +\infty}
\frac {|\gamma'(v)|^2}{\gamma(v)}$.  The global existence thereof in the higher dimensions has been proved for large $\mu > 0$ (\cite{WWang}),
while for small $\mu$, the respective  model can generate pattern formation (see \cite{MPWang}). The reader is referred to \cite{Lwang1,Lwang2}
 for the other studies on the related variants involving super-quadratic degradation terms.
%for similar results on related systems involving super-quadratic degradation terms.

In contexts of the  diffusion of  cells in a porous medium (see  the discussions in \cite{Calvez,Vazquez}), Winkler (\cite{WNon2})
 considered the cross--diffusion system
\begin{equation}\label{1.7}
\left\{
\begin{array}{ll}
u_{t}=\triangle(\gamma(v)u^m),\quad &x\in\Omega,t>0,\\
v_{t}=\Delta v-v+u,&x\in\Omega,t>0
\end{array}
\right.	
\end{equation}
in smoothly bounded convex domains $\Omega\subset \mathbb{R}^n$, where $ m> 1$, $ \gamma$ generalizes
the prototype $\gamma(v) =a+b(v+d)^{-\alpha} $ with $a\geq 0,b>0,d\geq 0$ and $\alpha\geq 0$, and proved the boundedness of global weak solutions
  to the associated initial-boundary value problem under some constriction on $m$ and $\alpha$, which particularly indicates that
  increasing $m$ in the cell equation goes along with a certain regularizing effect
despite both the diffusion and the cross-diffusion
mechanisms implicitly contained in  \eqref{1.7} are simultaneously enhanced.

In recent paper \cite{JSW}, Jin et al. considered the three-component system
%originated from (\ref{AHLsystem}) which can be specified as following
\begin{equation}\label{1.8}
\left\{
\begin{array}{ll}
u_{t}=\Delta (\gamma (v)u) +\beta uf(w)-\theta u,&x\in \Omega,~t>0,\\
v_{t}=D\Delta v+u-v,&x\in\Omega,~t>0,\\
w_{t}=\Delta w-uf(w),&x\in\Omega,~t>0
\end{array}
\right. 	
\end{equation}
 in a bounded domain $\Omega\subset \mathbb{R}^2$, where $\beta$,~$D>0$ and $\theta\geq 0$, the random motility function $\gamma (v)$ satisfies (H) and  functional
 response function $f(w)$ fulfills the assumption
 \begin{equation}\label{1.9}
	f(w)\in C^1([0,\infty)),~f(0)=0,  ~ f(w)>0~\textnormal{in}~(0,\infty)~ \textnormal{and}~f'(w)>0~\textnormal{on}~[0,\infty).
\end{equation}
Based on the method of energy estimates and Moser iteration, they
showed the uniformly boundedness %of classical solutions
to  initial--boundary value problem of \eqref{1.8}, inter alia  the asymptotic behavior thereof when parameter $D$ is suitably large.
In synopsis of the above results, one natural problem seems to consist
in determining to which extent nonlinear diffusion  of porous medium  type  may influence the solution behavior in chemotaxis systems involving density-suppressed motility. Accordingly, the purpose of the present work is to address this question
in the context of the particular choice $\gamma(v)=v^{-\alpha}$ with $\alpha>0$ instead of assumption (H) in \eqref{1.8}.
Specifically, we consider  the asymptotic behavior to  the initial--boundary value problem
\begin{equation}\label{1.10}
\left\{
\begin{array}{ll}
u_{t}=\Delta (\displaystyle\frac{u^m}{v^\alpha}) +\beta uf(w),&x\in \Omega,~t>0, \\
v_{t}=D\Delta v+u-v,&x\in\Omega,~t>0,\\
w_{t}=\Delta w-uf(w),&x\in\Omega,~t>0
\end{array}
\right. 	
\end{equation}
along with the initial
conditions
\begin{equation}\label{1.11}
u(x,0)=u_0,  v(x,0)=v_0~\and~ w(x,0)=w_0, ~~x\in \Omega
\end{equation}
and under the boundary conditions
\begin{equation}\label{1.12}
\displaystyle\frac{\partial u}{\partial \nu}=\displaystyle\frac{\partial v}{\partial \nu}
=\displaystyle\frac{\partial w}{\partial \nu}=0 ~~\hbox{on}~ \partial\Omega
\end{equation}
in a bounded convex domain $\Omega \subset \mathbb{R}^2$ with smooth boundary.

In what follows,  for simplicity we shall drop
the differential element in the integrals without confusion, namely abbreviating $\int_\Omega f(x) dx$ as $\int_\Omega f $
and  $\int^t_0\int_\Omega f(x,\tau) dxd\tau $ as $\int^t_0\int_\Omega f(\cdot,) d\tau$.  With the assumption \eqref{1.9},
our main result asserts that the weak solutions approach the relevant homogeneous steady state  in the large time limit
%with respect to $L^{\infty}$-norm
if $D$ is suitably large, which is stated as follows.
% we first prove the existence of
%globally bounded solutions to the system (1.3) in two dimensions as follows.
\begin{theorem}\label{Th1.1}
Let	$\Omega\subset \mathbb{R}^2$ be a bounded convex domain with smooth boundary, and suppose that $m>1,\alpha>0,\beta>0$ and $f$ satisfies \eqref{1.9}.
Assume that initial data $(u_0,v_0,w_0)\in (W^{1,\infty}(\Omega))^3 $ with $ u_0\gneqq  0 ,w_0\gneqq  0$ and $v_0>0$ in $\overline{\Omega}$. Then problem \eqref{1.10}--\eqref{1.12} admits at least one global  weak solution $(u,v,w )$ in the sense  of Definition 2.1 below. Moreover, there exists constant $D_0>0$ such that if $D>D_0$,
\begin{equation}\label{1.13}
\lim\limits_{t\rightarrow \infty}\|u(\cdot,t)-u_{\star}\|_{L^{\infty}(\Omega)}+\|v(\cdot,t)-u_{\star}\|_{L^{\infty}(\Omega)}+\|w(\cdot,t)\|_{L^{\infty}(\Omega)}=0	
\end{equation}
with $u_{\star}=\frac{1}{|\Omega|}\int_{\Omega}u_0+\frac{\beta}{|\Omega|}\int_{\Omega}w_0$.
\end{theorem}
As the first step to prove the above claim,  %investigate the asymptotic behavior of solutions,
in the next section it is shown  that problem \eqref{1.10}--\eqref{1.12} with $m>1$ and $\alpha>0$
possesses  a globally defined weak solution in two-dimensional setting  by adjusting the argument in \cite{WNon}.
With respect to the convergence properties asserted in \eqref{1.13}, our analysis is essentially different from that of \cite{JSW}. In fact, thanks to %thanks assuming that with the help of assumption
$ \gamma_1\leq\gamma(v)\leq \gamma_2$ for all $v\geq 0$ in (H), authors of \cite{JSW} derived the estimate of $\|u(\cdot,t)\|_{L^2(\Omega)}$, which is the start point of
a priori estimate of $\|u(\cdot,t)\|_{L^\infty(\Omega)}$. In particular, the assumption $ \gamma_1\leq\gamma(v)$  plays an essential role in constructing
energy function
$\mathcal{F}(u,v):=\|u(\cdot,t)-u_*\|_{L^2(\Omega)}+ \|v(\cdot,t)-u_*\|_{L^2(\Omega)}$, which leads to the convergence of $(u,v)$ if $D$ is suitable large (see the proofs of
Lemma 4.8 and Lemma  4.10 in \cite{JSW} for the details).   Our asymptotic analysis %will require substantially different ingredients. Indeed,
consists at its core in an analysis
of the functional
 $$
 \int_{\Omega}u^2+\eta\int_{\Omega}|\nabla v|^2
$$ for solutions of certain regularized versions of \eqref{1.10},
provided that in dependence on the model
parameter $D$ the positive constant $\eta$ is suitably chosen when $D$ is suitable large.
This yields
 the finiteness of $
\int^{\infty}_{3}\int_{\Omega}|\nabla u^{\frac{m+1}{2}}|^2
$
and
$
\int^{\infty}_{3}\int_{\Omega}|\nabla v
|^2
$
 (see Lemma \ref{lemma5.3}), and then entails
that as a consequence of these integral inequalities, all our solutions asymptotically
become homogeneous in space and hence satisfy \eqref{1.13} (Lemmas \ref{lemma5.4}--5.6).

\begin{remark}
(1) Note that as an apparently inherent drawback, assumption (H) in \cite{JSW} excludes $\gamma(v)$ decay functions  such as  $v^{-\alpha}$.  Indeed, despite $v$ is bounded below by $\delta$ with the help of
  Lemma \ref{Lemma2.3} and thereby the upper bound for  $\gamma(v)$  can be removed,  an lower bound for  $\gamma(v)$ in (H) is essentially required therein.

(2) Due to the  results on existence of global solutions in \cite{WNon2}, the asymptotic behavior of solutions herein seems to be achieved for the higher-dimensional version of \eqref{1.10}  at the cost of additional constraint on $m$ and $\alpha$.
\end{remark}

\section{Preliminaries}
Throughout this paper, we shall pursue weak solutions to problem \eqref{1.10}--\eqref{1.12} specified as follows.
\begin{definition}\label{Definition 2.1}
Let $m>1,\alpha>0,\beta>0$ and $f$ satisfies \eqref{1.9}. Then  %and that $u_0$,$v_0$ and $w_0$ are nonnegative functions from $L^1(\Omega)$.~Then
a triple $(u,v,w) $ of nonnegative functions %a group of functions.
$$
\left\{
\begin{array}{lll}
u\in L^1_{loc}(\overline{\Omega}\times[0,\infty))\\
v\in L^1_{loc}([0,\infty);W^{1,1}(\Omega))\\
w\in L^1_{loc}([0,\infty);W^{1,1}(\Omega))
\end{array}
\right.
$$
will be called a global weak solution of problem \eqref{1.10}--\eqref{1.12}  if %$u\geq 0$,~$v>0$,~$w\geq 0$ \emph{a.e.} in $\Omega\times(0,\infty)$,	
\begin{equation}\label{2.1}
u^m
/v^\alpha\in L^1_{loc}(\overline{\Omega}\times[0,\infty))
\end{equation}
and
\begin{equation}\label{2.2}
-\int^{\infty}_0\int_{\Omega}u\varphi_t-\int_{\Omega}u_0\varphi(\cdot,0)=\int^{\infty}_0\int_{\Omega}\frac{u^m}{v^\alpha}\Delta\varphi+
\beta\int^{\infty}_0\int_{\Omega}uf(w)\varphi	
\end{equation}
for all $\varphi\in C^{\infty}_0(\overline{\Omega}\times[0,\infty))$ such that $\frac{\partial\varphi}{\partial \nu}|_{\partial \Omega}=0$ and
\begin{equation}\label{2.3a}
-\int^{\infty}_0\int_{\Omega}v\varphi_t-\int_{\Omega}v_0\varphi(\cdot,0)=-D\int^{\infty}_0\int_{\Omega}\nabla v\cdot\nabla \varphi-\int^{\infty}_0\int_{\Omega}v\varphi+\int^{\infty}_0\int_{\Omega}u\varphi	
\end{equation}
for all $\varphi\in C^{\infty}_0(\overline{\Omega}\times[0,\infty))$ as well as
\begin{equation}\label{2.4}
\int^{\infty}_0\int_{\Omega}	 w\varphi_t-\int_{\Omega}w_0\varphi(\cdot,0)=-\int^{\infty}_0\int_{\Omega}\nabla w\cdot\nabla \varphi-\int^{\infty}_0\int_{\Omega}uf(w)\varphi
\end{equation}
for all $\varphi\in C^{\infty}_0(\overline{\Omega}\times[0,\infty))$.
\end{definition}

Proceeding  in a similar manner as done  in \cite{WNon2}, a global weak solution in the above sense can be obtained as the limit of a
sequence of solutions $(u_{\varepsilon},v_{\varepsilon},w_{\varepsilon}) $ of the regularized problems
\begin{equation}\label{2.5}
\left\{
\begin{array}{lll}
u_{\varepsilon t}=\varepsilon\Delta (u_{\varepsilon}+1)^{M}+\Delta \left(u_{\varepsilon}(u_{\varepsilon}+\varepsilon )^{m-1}v_{\varepsilon}^{-\alpha}\right) +\beta u_{\varepsilon}f(w_{\varepsilon}),&&x\in \Omega,~t>0, \\
v_{\varepsilon t}=D\Delta v_{\varepsilon}+u_{\varepsilon}-v_{\varepsilon},&&x\in\Omega,~t>0,\\
w_{\varepsilon t}=\Delta w_{\varepsilon}-u_{\varepsilon}f(w_{\varepsilon}),&&x\in\Omega,~t>0,	\\
\frac{\partial u_{\varepsilon}}{\partial \nu}=\frac{\partial v_{\varepsilon}}{\partial \nu}=\frac{\partial w_{\varepsilon}}{\partial \nu}=0,&&x\in\partial \Omega,~t>0,\\
u_{\varepsilon}(x,0)=u_0,~v_{\varepsilon}(x,0)=v_0,~w_{\varepsilon}(x,0)=w_0,&&x\in \Omega
\end{array}
\right. 	
\end{equation}
with $M>m,\varepsilon\in(0,1)$.  %The proof thereof is omitted here for the of presentation.	
\begin{lemma}\label{Lemma2.1}
Let $m>1,\alpha>0,\beta>0$ and $f$ satisfies \eqref{1.9}.  Then there exist $(\varepsilon _{j})_{j\in\mathbb{N}}\subset(0,1)$ as well as nonnegative functions
\begin{equation}\label{2.6}
\left\{
\begin{array}{lll}
u\in L^{\infty}(\overline{\Omega}\times[0,\infty))	\\
v\in C^0(\overline{\Omega}\times[0,\infty))\bigcap L^2_{loc}([0,\infty);W^{1,2}(\Omega))\\
w\in C^0(\overline{\Omega}\times[0,\infty))\bigcap L^2_{loc}([0,\infty);W^{1,2}(\Omega))
\end{array}
\right.	
\end{equation}
such that $\varepsilon_{j}\searrow 0$ as $j\rightarrow \infty$ and as $\varepsilon_{j}\searrow 0$, we have
\begin{eqnarray}\label{2.7}
&&u_{\varepsilon}\rightarrow u\quad \textnormal{\emph{a.e. in $\Omega\times(0,\infty)$}},\\
&&u_{\varepsilon}\rightarrow u\quad \textnormal{\emph{in $\bigcap_{p\geq 1}L^p_{loc}(\overline{\Omega}\times[0,\infty))$ }},\\
&&v_{\varepsilon}\rightarrow v\quad\textnormal{\emph{in $C_{loc}^0(\overline{\Omega}\times[0,\infty))$}},\\
&&w_{\varepsilon}\rightarrow w\quad\textnormal{\emph{in $C_{loc}^0(\overline{\Omega}\times[0,\infty))$}},\\
&&\nabla v_{\varepsilon}\rightharpoonup \nabla v \quad\textnormal{\emph{in $L^2_{loc}(\overline{\Omega}\times[0,\infty))$}},\\
&&\nabla w_{\varepsilon}\rightharpoonup \nabla w \quad\textnormal{\emph{in $L^2_{loc}(\overline{\Omega}\times[0,\infty))$}}	.
\end{eqnarray}
Moreover,~$v>0$ in $\overline{\Omega}\times(0,\infty)$ and $(u,v,w)$ forms a global weak solution of \eqref{1.10}--\eqref{1.12}
 in the sense of Definition \ref{Definition 2.1}.
\end{lemma}

The following basic properties of the spatial $L^1$ norms of $(u_\varepsilon,v_\varepsilon,w_\varepsilon)$ as well as the $L^\infty$ norm of $w_\varepsilon$ are easily verified.

\begin{lemma}\label{Lemma2.2} Let $(u_{\varepsilon},v_{\varepsilon},w_{\varepsilon})$ be the classical solution of \eqref{2.5} in $\Omega\times(0,\infty)$. Then we have 	
\begin{equation}\label{2.13}
\|u_{\varepsilon}(\cdot,t)\|_{L^1(\Omega)}+\beta\|w_{\varepsilon}(\cdot,t)\|_{L^1(\Omega)}=\|u_0\|_{L^1(\Omega)}+\beta\|w_0\|_{L^1(\Omega)},	
\end{equation}
\begin{equation}\label{2.14}
\|u_{\varepsilon}(\cdot,t)\|_{L^1(\Omega)}
\geq \|u_0\|_{L^1(\Omega)},
\end{equation}
\begin{equation}\label{2.15}
\int_{\Omega}v_{\varepsilon}(\cdot,t)\leq \int_{\Omega} v_0+\int_{\Omega}u_0+\beta\int_{\Omega}w_0	 %\quad\textnormal{\emph{for all $\varepsilon\in(0,1)$, $t>0$}}
\end{equation}
as well as
\begin{equation}\label{2.16}
t\mapsto \|w_{\varepsilon}(\cdot,t)\|_{L^{\infty}(\Omega)} ~\textnormal{\emph{is nonincrasing in}}~~[0,\infty).
\end{equation}
\end{lemma}
\begin{proof}
Multiplying  $w_{\varepsilon}$-equation by $\beta$ and adding the result  to $u_{\varepsilon}$-equation in \eqref{2.5}, we  get
\begin{equation}\label{2.17}
\beta \frac{d}{dt}\int_{\Omega}w_{\varepsilon}+\frac{d}{dt}\int_{\Omega}u_{\varepsilon }=0,
\end{equation}
which immediately yields \eqref{2.13}.  An integration of the first equation in \eqref{2.5} gives us
\begin{equation}\label{2.18}
\frac{d}{dt}\int_{\Omega}u_{\varepsilon}=\int_{\Omega}u_{\varepsilon}f(w_{\varepsilon})\geq 0	
\end{equation}
which readily entails \eqref{2.14}.
Upon the integration of the second equation in \eqref{2.5}, we can see that
$$
\frac{d}{dt}\int_{\Omega}v_{\varepsilon}+\int_{\Omega}v_{\varepsilon}\leq \int_{\Omega}u_{\varepsilon}
$$
which, along with \eqref{2.13} leads to \eqref{2.15}.
Due to the fact that $f$ and  $w_{\varepsilon}$ are nonnegative, the claim in \eqref{2.16} results upon an
application of the maximum principle to $w_{\varepsilon}$-equation in \eqref{2.5}.
\end{proof}

%%Finally,~let us formulate a $\varepsilon$-independent estimation of $v_{\varepsilon}$ through a straightforward semigroup argument.

Let us  first derive lower bound for $v_\varepsilon$ which will
alleviate the difficulties caused by the singularity of  signal-dependent motility function $v^{-\alpha}$ near zero. Despite the quantitative lower estimate for solutions of the Neumann problem was established in the related literature (\cite{HPW,WNon2}),  we  present a proof of our results with some necessary details  to  make the lower bound accessible to the sequel analysis.

\begin{lemma}\label{Lemma2.3}
If $D\geq 1$, then there exist a constant $\delta>0$ independent of $D$ such that
\begin{equation}\label{2.10}
v_{\varepsilon}(x,t)>\delta	
\end{equation}
for all $x\in \Omega$ and $t>2$.
\end{lemma}

\begin{proof}
  According to the pointwise lower bound estimate for the Neumann heat semigroup $(e^{t\Delta})_{t\geq 0}$ on the convex domain $\Omega$,
one can find $c_1(\Omega)>0$ such that
$$
e^{t\Delta}\varphi\geq c_1(\Omega)\int_{\Omega}\varphi \qquad \textnormal{for all $t\geq 1$ and each nonnegative $\varphi\in C^0(\overline{\Omega})$}
$$
(e.g.~\cite{Fujie,HPW}).

By the time rescaling $\tilde t=Dt$, we can see that $\tilde{v}(x,\tilde t):=v_{\varepsilon}(x,\frac{\tilde t} D)$ satisfies
\begin{equation}\label{2.20a}
\tilde{v}_{\tilde{t}}=\Delta \tilde{v}-D^{-1}\tilde{v}+D^{-1} u_{\varepsilon}(x,D^{-1}\tilde t).
\end{equation}Now applying the variation-of-constant formula to \eqref{2.20a}, we have
\begin{equation}\label{2.21a}
\begin{array}{ll}
\tilde v (\cdot,\tilde t)& =e^{\tilde t(\Delta -D^{-1})}v_0(\cdot)+
D^{-1}\displaystyle \int^{\tilde t}_{0}e^{(\tilde t-s)(\Delta-D^{-1})}u_{\varepsilon}(\cdot,D^{-1}s)ds\\
& \geq D^{-1}\displaystyle\int^{\tilde t-1}_{0}e^{(\tilde t-s)(\Delta-D^{-1})}u_{\varepsilon}(\cdot,D^{-1}s)ds\\
&\geq c_1(\Omega) D^{-1}(\displaystyle\int^{\tilde t-1}_0 e^{-D^{-1}(\tilde t-s)}ds) \inf\limits_{s\in(0,\infty)}\int_{\Omega}u_{\varepsilon}(\cdot,s)\\
&\geq c_1(\Omega) (e^{-D^{-1}}-e^{-D^{-1}\tilde t})\displaystyle\int_{\Omega}u_{0}
\end{array}
\end{equation}
for all $\tilde t>2$.
Hence due to $D\geq 1$, we  can see that for $x\in \Omega$ and $\tilde t\geq 2D$
\begin{equation*}
\tilde v(x,\tilde t)\geq  \frac{c_1(\Omega)}{2e}\int_{\Omega}u_{0},	
\end{equation*}
and readily establish \eqref{2.10} with   $\delta= \frac{c_1(\Omega)}{2e}\int_{\Omega}u_{0}	$.
\end{proof}

Through a straightforward semigroup argument, we formulate a favorable  dependence of $\|v_{\varepsilon}(\cdot,t)\|_{L^{p}(\Omega)}$
with respect to  parameter  $D$.

\begin{lemma}\label{lemma2.4}
For $p>1$, there exists $C(p)>0$ such that
\begin{equation}\label{2.13a}
\|v_{\varepsilon}(\cdot,t)\|_{L^{p}(\Omega)}\leq C(p)(1+D^{\frac 1p-1}	)
\end{equation}
for all $t>2$.
\end{lemma}
\begin{proof} Applying a Duhamel's formula to \eqref{2.20a} and employing well-known smoothing properties of the Neumann heat semigroup $(e^{t\Delta})_{t\geq 0}$ on $\Omega$ (see  Lemma 3  of \cite{Rothe} or Lemma 1.3   of \cite{WJDE} for example),  we can  find $c_p>0$ such that for any $\tilde t\geq 2D $
\begin{equation*}
\begin{array}{rl}
& \|\tilde v_{\varepsilon}(\cdot,\tilde t)\|_{L^p(\Omega)}\\
 = &\|e^{-D^{-1}\tilde t} e^{\tilde t \Delta } v_0(\cdot)+D^{-1}\displaystyle \int^{\tilde t}_{ 0}e^{(\tilde t-s)(\Delta-D^{-1}) }u_{\varepsilon}(\cdot,D^{-1}s)ds\|_{L^p(\Omega)}\\
\leq & 	c_p e^{-D^{-1} \tilde t}(1+\tilde t^{-1+\frac 1p})\| v_0\|_{L^1(\Omega)}
+\displaystyle\frac {c_p} {D}\displaystyle \int^{\tilde t}_{ 0}e^{-D^{-1}(\tilde t-s) }(1+(\tilde t-s)^{-1+\frac 1p})\|u_{\varepsilon}(\cdot,D^{-1}s)\|_{L^1(\Omega)}ds\\
\leq & 	2c_p e^{-D^{-1} \tilde t}\| v_0\|_{L^1(\Omega)}
+\displaystyle\frac {c_p} {D}(\|u_0\|_{L^1(\Omega)}+\beta\|w_0\|_{L^1(\Omega)})	\displaystyle \int^{\tilde t}_{ 0}e^{-D^{-1}(\tilde t-s) }(1+(\tilde t-s)^{-1+\frac 1p}) ds\\
= & 	2c_p e^{-D^{-1} \tilde t}\| v_0\|_{L^1(\Omega)}
+\displaystyle\frac {c_p} {D}(\|u_0\|_{L^1(\Omega)}+\beta\|w_0\|_{L^1(\Omega)})	\displaystyle \int^{\tilde t}_{ 0}e^{-D^{-1}\sigma}(1+\sigma^{-1+\frac 1p}) d\sigma\\
\leq & 	2c_p e^{-D^{-1} \tilde t}\| v_0\|_{L^1(\Omega)}
+\displaystyle\frac {c_p} {D}(\|u_0\|_{L^1(\Omega)}+\beta\|w_0\|_{L^1(\Omega)})(D+D^{\frac 1 p}	\displaystyle \int^{\infty}_{0}e^{-\sigma}\sigma^{-1+\frac 1p} d\sigma)\\
\leq & 	2c_p \| v_0\|_{L^1(\Omega)}
+ (1+D^{\frac 1 p-1})c_p (\|u_0\|_{L^1(\Omega)}+\beta\|w_0\|_{L^1(\Omega)})(1+\displaystyle \int^{\infty}_{0}e^{-\sigma}\sigma^{-1+\frac 1p} d\sigma)
\end{array}
\end{equation*}
which
  %Since $\displaystyle\lim_{\tilde t\rightarrow \infty} e^{-D^{-1} \tilde t}=0$, there exists $t_0(D,p)>2$ such that
  ends up \eqref{2.13a} with  $C(p)= 3 c_p (\|u_0\|_{L^1(\Omega)}+\beta\|w_0\|_{L^1(\Omega)})
  (1+\displaystyle \int^{\infty}_{0}e^{-\sigma}\sigma^{-1+\frac 1p} d\sigma) $.
\end{proof}

\section{Space-time $L^1$-estimates for $u^{m+1}_{\varepsilon}v^{-\alpha}_{\varepsilon}$}
 In this section, taking advantage of special structure of the diffusive processes in \eqref{2.5} (also \eqref{1.8}),
 the classical duality arguments (cf. \cite{TWM3AS, CDF})
% \cite{Winkler,} %for some variants of (\ref{MainSystem}) with two component setting)
 is used  to obtain the fundamental regularity information for a bootstrap argument.
 To this end, we denote $A$ to the self-adjoint realization of $-\Delta +1$ under homogeneous Neumann boundary condition in
  $L^2(\Omega)$ with its domain given by $D(A)=\left\{\varphi \in W^{2,2}(\Omega)|\frac{\partial\varphi}{\partial \nu}=0\right\}$ and $A$ is self-adjoint and possesses a
 family $(A^{\beta})_{\beta \in \mathbb{R}}$ of corresponding densely defined self-adjoint fractional powers.

\begin{lemma}\label{lemma3.1}
Assume that $m> 1$ and $D\geq 1$, then  for all $t>2$
\begin{equation}\label{3.1}
\frac{d}{dt}\int _{\Omega}|A^{-\frac{1}{2}}(u_{\varepsilon}+1)|^2+\int _{\Omega}u^{m+1}_{\varepsilon}v^{-\alpha}_{\varepsilon}\leq C\int_{\Omega}|A^{-1}(u_{\varepsilon}+1)|^{m+1}+C
\end{equation}
 with  constant $C>0$ independent of $D$.
\end{lemma}	
\begin{proof}
Due to $\partial _t(u_{\varepsilon}+1)=u_{\varepsilon t}$, the first equation in \eqref{2.5} can be  written as
\begin{equation}\label{3.2}
\begin{array}{rl}
&\displaystyle\frac{d}{dt} A^{-1}(u_{\varepsilon}+1)+\varepsilon(u_{\varepsilon}+1)^{M} +u_{\varepsilon}(u_{\varepsilon}+\varepsilon)^{m-1} v_\varepsilon ^{-\alpha}\\
=&A^{-1}\left\{\varepsilon(u_{\varepsilon}+1)^M+u_{\varepsilon}(u_{\varepsilon}+\varepsilon)^{m-1}v_\varepsilon ^{-\alpha}+  \beta u_{\varepsilon}
f(w_{\varepsilon})\right\}.	
\end{array}
\end{equation}
Testing \eqref{3.2} by $u_{\varepsilon}+1$, one has
\begin{equation}\label{3.3}
\begin{array}{rl}
&\displaystyle\frac{1}{2}\frac{d}{dt}\int _{\Omega}|A^{-\frac{1}{2}}(u_{\varepsilon}+1)|^2+\varepsilon\int_{\Omega}(u_{\varepsilon}+1)^{M+1}+\int_{\Omega}
u_{\varepsilon}(u_{\varepsilon}+\varepsilon)^{m-1}(u_{\varepsilon}+1)v_\varepsilon ^{-\alpha}\\
=&\varepsilon\displaystyle\int_{\Omega}(u_{\varepsilon}+1)^M A^{-1}(u_{\varepsilon}+1)+\int_{\Omega}u_{\varepsilon}(u_{\varepsilon}+\varepsilon)^{m-1}v_\varepsilon^{-\alpha} A^{-1}(u_{\varepsilon}+1)+
\beta \int_{\Omega} u_{\varepsilon} f(w_{\varepsilon})A^{-1}(u_{\varepsilon}+1).
\end{array}
\end{equation}
Thanks to $W^{2,2}(\Omega)\hookrightarrow L^{\infty}(\Omega)$ in two-dimensional setting and  the standard elliptic regularity in $L^2(\Omega)$,
one can find $c_1>0$ and $c_2>0$ such that
\begin{equation}\label{3.4}
\|\varphi\|^{M+1}_{L^{M+1}(\Omega)}\leq c_1\|\varphi\|^{M+1}_{W^{2,2}(\Omega)}\leq c_2\|A\varphi\|^{M+1}_{L^2(\Omega)}	
\end{equation}
for all $\varphi\in W^{2,2}(\Omega)$ such that $\frac{\partial \varphi}{\partial\nu}|_{
\partial \Omega }=0$.
 Hence by the Young inequality, we can see that
\begin{equation}\label{3.5}
\begin{array}{rl}
\varepsilon\displaystyle\int_{\Omega}(u_{\varepsilon}+1)^MA^{-1}(u_{\varepsilon}+1)\leq &
\displaystyle\frac{\varepsilon}{2}\int_{\Omega}(u_{\varepsilon}+1)^{M+1}+
\displaystyle\frac{\varepsilon}{2}\int_{\Omega}|A^{-1}(u_{\varepsilon}+1)|^{M+1}\\
\leq & \displaystyle\frac{\varepsilon}{2} \|u_{\varepsilon}
+1\|^{M+1}_{L^{M+1}(\Omega)}+ \displaystyle\frac{\varepsilon c_1}{2}\|A^{-1}(u_{\varepsilon}+1)\|^{M+1}_{W^{2,2}(\Omega)}\\
=&\displaystyle\frac{\varepsilon}{2} \displaystyle \int_{\Omega}(u_{\varepsilon}+1)^{M+1}+\displaystyle\frac{\varepsilon c_1 c_2}{2}\|u_{\varepsilon}+1\|^{M+1}_{L^2(\Omega)}
\end{array}
\end{equation}
which,  along with the interpolation inequality that for any $\varepsilon_1>0$, there exits $c(\varepsilon_1)>0$ such that
$\|\varphi\|_{L^2(\Omega)}\leq \varepsilon_1\|\varphi\|_{L^{M+1}(\Omega)}+c(\varepsilon_1)\|\varphi\|_{L^1(\Omega)}$ due to $M>1$, entails that
\begin{equation}\label{3.6}
\varepsilon\displaystyle\int_{\Omega}(u_{\varepsilon}+1)^MA^{-1}(u_{\varepsilon}+1)\leq \displaystyle\frac{3\varepsilon}{4} \displaystyle \int_{\Omega}(u_{\varepsilon}+1)^{M+1}+c_3\|u_{\varepsilon}+1\|^{M+1}_{L^1(\Omega)}
\end{equation}
Furthermore, since  $\|w_\varepsilon(\cdot,t)\|_{L^\infty(\Omega)}\leq \|w_0\|_{L^\infty(\Omega)} $, we  use Lemma \ref{Lemma2.1} and Young's inequality to get that for $t>2$,
\begin{equation}
\begin{array}{rl}\label{3.7}
&\displaystyle\int_{\Omega}u_{\varepsilon}(u_{\varepsilon}+\varepsilon)^{m-1} v_\varepsilon ^{-\alpha} A^{-1}(u_{\varepsilon}+1)
\\[2mm]
\leq &\displaystyle\frac 14 \displaystyle\int_{\Omega}\left\{u_{\varepsilon}(u_{\varepsilon}+\varepsilon)^{m-1}\right\}^{\frac{m+1}{m}}
v_\varepsilon ^{-\alpha}+
c_4\int_{\Omega}|A^{-1}(u_{\varepsilon}+1)|^{m+1}v_\varepsilon ^{-\alpha}\\[2mm]
\leq &\displaystyle\frac{1}{4}\int_{\Omega}u^{\frac{m+1}{m}}_{\varepsilon}(u_{\varepsilon}+\varepsilon)^{\frac{m^2-1}{m}}v_\varepsilon ^{-\alpha}
+c_4\delta ^{-\alpha} \int_{\Omega}|A^{-1}(u_{\varepsilon}+1)|^{m+1}	
\end{array}
\end{equation}
and
\begin{equation}\label{3.8}
\begin{array}{rl}
&\beta \displaystyle\int_{\Omega} u_{\varepsilon}f(w_{\varepsilon})A^{-1}(u_{\varepsilon}+1)\\[3mm]
\leq & \displaystyle\frac 14 \int_{\Omega}u_{\varepsilon}^{m+1} v_\varepsilon^{-\alpha} +c_5\int_{\Omega}v_{\varepsilon}^{\frac{\alpha} m}|A^{-1}(u_{\varepsilon}+1)|^{\frac{m+1}{m}}  \\
\leq &\displaystyle\frac{1}{4}\int_{\Omega}u^{m+1}_{\varepsilon}
v_\varepsilon^{-\alpha}
+
\int_{\Omega}|A^{-1}(u_{\varepsilon}+1)|^{m+1}+c_6\int_{\Omega}v_{\varepsilon}^{\frac{\alpha}{m-1}}.\end{array}
\end{equation}

Noticing that  $u_{\varepsilon}+1\geq \max\{u_{\varepsilon}+\varepsilon,\varepsilon\}$, we have
$$
\int_{\Omega}u_{\varepsilon}(u_{\varepsilon}+\varepsilon)^{m-1}(u_{\varepsilon}+1)v_\varepsilon^{-\alpha}\geq \frac{1}{4}\int_{\Omega}u_{\varepsilon}^{\frac{m+1}{m}}(u_{\varepsilon}+\varepsilon)^{\frac{m^2-1}{m}}v_\varepsilon^{-\alpha}+ \frac{3}{4}\int_{\Omega}u_{\varepsilon}^{m+1}v_\varepsilon^{-\alpha},	
$$
and hence insert  \eqref{3.8} and  \eqref{3.7} into \eqref{3.3} to get
\begin{equation}
\begin{array}{rl}
&\displaystyle\frac{d}{dt}\int _{\Omega}|A^{-\frac{1}{2}}(u_{\varepsilon}+1)|^2+ \int_{\Omega}u_{\varepsilon}^{m+1}v_{\varepsilon}^{-\alpha}\nonumber
\\
\leq  & 2(c_4\delta ^{-\alpha}+1)\displaystyle\int_{\Omega}|A^{-1}(u_{\varepsilon}+1)|^{m+1}+2c_6\int_{\Omega}v_{\varepsilon}^{\frac{\alpha}{m-1}},
\end{array}
\end{equation}
which along with Lemma \ref{lemma2.4} and $D\geq 1$ readily arrive at \eqref{3.1}.
\end{proof}

By means of suitable interpolation arguments,  one can  appropriately estimate the integrals $\int_{\Omega}|A^{-1}(u_{\varepsilon}+1)|^{m+1}$ and
$\int _{\Omega}|A^{-\frac{1}{2}}(u_{\varepsilon}+1)|^2$ in  terms of $ \int_{\Omega}u_{\varepsilon}^{m+1}v_{\varepsilon}^{-\alpha} $ and thereby derive estimate of the form
 $$
 \int^{t+1}_t\int_{\Omega}u^{m+1}_{\varepsilon}v^{-\alpha}_{\varepsilon}\leq C
 $$ with $C>0$ independent of $D$, which can be stated as follows

\begin{lemma}\label{lemma3.2}
If $m>1$ and $D\geq 1$, then there exists $C>0 $  independent of $D$  such that
\begin{equation}\label{3.9}
\int^{t+1}_t\int_{\Omega}u^{m+1}_{\varepsilon}v^{-\alpha}_{\varepsilon}\leq C~~~ ~~\hbox{for all}~~ t>2.
\end{equation}
%with  constant $C>0$  independent of $D$.
\end{lemma}
\begin{proof} By the standard elliptic regularity in $L^2(\Omega)$, we have
$$\displaystyle\int_{\Omega}|A^{-1}(u_{\varepsilon}+1)|^{m+1}\leq c_1\displaystyle\|u_{\varepsilon}+1\|_{L^2(\Omega)}^{m+1}.
$$
Noticing that for the  given $p\in (2,m+1)$ (for example $p:=\frac {m+3}2$), an application of the interpolation inequality implies that for any $\eta>0$, there exists $c_1(\eta)>0$ such that
$$c_1\|u_{\varepsilon}+1\|^{m+1}_{L^2(\Omega)}\leq \eta \|u_{\varepsilon}+1\|^{m+1}_{L^p(\Omega)} +c_1(\eta)\|u_{\varepsilon}+1\|^{m+1}_{L^1(\Omega)}.$$
On the other hand,  by the H\"{o}lder inequality, we can see that
\begin{eqnarray*}
\int_{\Omega}u^{p}_{\varepsilon}&=&\int_{\Omega}\left(u^{m+1}_{\varepsilon}v_{\varepsilon}^{-\alpha}\right)^{\frac{p}{m+1}}v_{\varepsilon}^{\frac{p\alpha}{m+1}}\nonumber	 \\
&\leq & (\int_{\Omega}u^{m+1}_{\varepsilon}v^{-\alpha}_{\varepsilon})^{\frac p{m+1}}
(\int_{\Omega}v^{\frac{p\alpha}{m+1-p}}_{\varepsilon})^{\frac {m+1-p}{m+1}}.
\end{eqnarray*}
Hence combining  above estimates with Lemma \ref{lemma2.4},  we arrive at
\begin{equation}\label{3.10}
\begin{array}{rl}
\displaystyle\int_{\Omega}|A^{-1}(u_{\varepsilon}+1)|^{m+1}
& \leq \eta\displaystyle\|u_{\varepsilon}+1\|_{L^p(\Omega)}^{m+1}
+c_1(\eta)\|u_{\varepsilon}+1\|^{m+1}_{L^1(\Omega)}\\
& \leq \eta\displaystyle\|u_{\varepsilon}\|_{L^p(\Omega)}^{m+1}+c_2(\eta)\\
& \leq  \displaystyle\eta(\int_{\Omega}u^{m+1}_{\varepsilon}v^{-\alpha}_{\varepsilon})
(\int_{\Omega}v^{\frac{p\alpha}{m+1-p}}_{\varepsilon})^{\frac {m+1-p}{p}}+c_2(\eta)\\
& \leq  \displaystyle\eta c_3(\alpha,m)(\int_{\Omega}u^{m+1}_{\varepsilon}v^{-\alpha}_{\varepsilon})+c_2(\eta)~~~\hbox{for all}~~t>2.
\end{array}
\end{equation}
On the other hand, by self-adjointness of $A^{-\frac 12}$ and  H\"{o}lder's inequality, we get

\begin{equation}\label{3.11}
\begin{array}{rl}
\displaystyle\int_{\Omega}|A^{-\frac 12}(u_{\varepsilon}+1)|^{2}=&
\displaystyle\int_{\Omega}(u_{\varepsilon}+1) A^{-1}(u_{\varepsilon}+1)\\
\leq & \displaystyle\|u_{\varepsilon}+1\|_{L^2(\Omega)}
\|A^{-1}(u_{\varepsilon}+1)\|_{L^2(\Omega)}\\
\leq & c_4\displaystyle\|u_{\varepsilon}+1\|^2_{L^2(\Omega)}\\
\leq & c_5\displaystyle\|u_{\varepsilon}\|^2_{L^{2}(\Omega)}+c_5\\
\leq & \displaystyle c_6\|u_{\varepsilon}\|^{m+1}_{L^{2}(\Omega)}+c_6.
\end{array}
\end{equation}
So in this position, proceeding in the same way as above we also have
\begin{equation}\label{3.12}
\displaystyle\int_{\Omega}|A^{-\frac 12}(u_{\varepsilon}+1)|^{2}
\leq \displaystyle\eta c_3(\alpha,m)(\int_{\Omega}u^{m+1}_{\varepsilon}v^{-\alpha}_{\varepsilon})+c_7(\eta).
\end{equation}
Therefore inserting \eqref{3.10} and \eqref{3.12} into \eqref{3.1} and taking $\eta$ sufficiently small, we have
\begin{equation}\label{3.13}
\frac{d}{dt}\int _{\Omega}|A^{-\frac{1}{2}}(u_{\varepsilon}+1)|^2+  c_8\int _{\Omega}|A^{-\frac{1}{2}}(u_{\varepsilon}+1)|^2 +c_8 \int _{\Omega}u^{m+1}_{\varepsilon}v^{-\alpha}_{\varepsilon}\leq c_9
\end{equation}
for  some $c_8>0, c_9>0$ which are  independent of $D$. Furthermore, by  Lemma 3.4 of \cite{SSW} , we immediately obtain \eqref{3.9}.
\end{proof}
 As the direct consequence of Lemma \ref{lemma3.2} and  Lemma \ref{lemma2.4}, we have
\begin{lemma}\label{lemma3.3}
Let $m>1,D\geq 1$,  then for $p\in (\max\{2,\frac {m+1}{\alpha+1}\}, m+1)$ one can find  a constant $C(p)>0$ independent of $D$ such that
\begin{equation}\label{3.14}
\int^{t+1}_{t}\int_{\Omega}u^{p}_{\varepsilon}(\cdot,s)ds\leq C(p)   ~ ~~\hbox{for all}~~ t>2. %\quad \textnormal{\emph{for all $t>t_0$ and $\varepsilon\in (0,1)$}},	
\end{equation}
\end{lemma}
\begin{proof}
For $p\in (2,m+1)$, we utilize Young's inequality to estimate
\begin{equation*}
\begin{array}{rl}
\displaystyle\int^{t+1}_{t}\int_{\Omega}u^{p}_{\varepsilon}&=\displaystyle\int ^{t+1}_{t}\int_{\Omega}\left(u^{m+1}_{\varepsilon}v_{\varepsilon}^{-\alpha}\right)^{\frac{p}{m+1}}v_{\varepsilon}^{\frac{p\alpha}{m+1}} \\
&\leq \displaystyle\int^{t+1}_{t}\int_{\Omega}u^{m+1}_{\varepsilon}v^{-\alpha}_{\varepsilon}+\int ^{t+1}_{t}\int_{\Omega}v^{\frac{p\alpha}{m+1-p}}_{\varepsilon},
\end{array}
\end{equation*}
which leads to \eqref{3.14} with the help of Lemma \ref{lemma2.4}.
\end{proof}

\section{Boundedness of solutions  $(u_{\varepsilon}, v_{\varepsilon}, w_{\varepsilon})$}

On the basis of the quite well established arguments from parabolic regularity theory, we can turn  the space--time integrability properties of $u^{p}_{\varepsilon}$ into the integrability properties of $\nabla v_{\varepsilon}$ as well as  $\nabla w_{\varepsilon}$.

\begin{lemma}\label{lemma4.1}
Let $m>1,\alpha>0 $ and suppose that $D\geq 1$.  Then for $q\in (2,\frac {2(m+1)}{(3-m)_+})$, there exist constant $C>0 $  independent of $D$  such that
\begin{equation}\label{4.1}
	\|v_{\varepsilon}(\cdot,t)\|_{W^{1,q}(\Omega)}\leq C
\end{equation}
as well as
\begin{equation}\label{4.2}
\|w_{\varepsilon}(\cdot,t)\|_{W^{1,q}(\Omega)}\leq C
\end{equation}
for  all $t>2$.
\end{lemma}

\begin{proof}
From the continuity of function $h(x)=\frac{2x}{(4-x)_+}$,  it follows that  for given $q>2$ suitably close to the number $\frac {2(m+1)}{(3-m)_+}$,  one can choose  $p\in (2,m+1)$ in an appropriately small neighborhood of $m+1$ such that
\begin{equation}\label{4.3}
\frac{p}{p-1}\cdot(\frac{1}{2}+\frac{1}{p}-\frac{1}{q})<1.		
\end{equation}

From the smoothing properties of Neumann heat semigroup $(e^{t\Delta})_{t\geq 0}$, it follow that there exist $c_i>0 (i=1,2) $
such that
%\begin{equation}\label{4.4}
%\|e^{t\Delta}\varphi\|_{W^{1,q}(\Omega)}\leq c_1\|\varphi\|_{W^{1,\infty}(\Omega)}\quad \textnormal{\emph{for all $t\in (0,1)$ and $\varphi\in W^{1,\infty}(\Omega)$}}	,
%\end{equation}
\begin{equation}\label{4.4}
\|e^{\Delta}\varphi\|_{W^{1,q}(\Omega)}\leq c_1\|\varphi\|_{L^1(\Omega)}\quad \textnormal{\emph{for all $t\in (0,1)$ and $\varphi\in C^0(\overline{\Omega})$}}	 
\end{equation}
as well as
\begin{equation}\label{4.5}
\|e^{t\Delta}\varphi\|_{W^{1,q}(\Omega)}\leq c_2 t^{-\frac{1}{2}-\frac{1}{2}(\frac{1}{p}-\frac{1}{q})}\|\varphi\|_{L^p(\Omega)} \quad \textnormal{\emph{for all $t\in (0,1)$ and $\varphi\in C^0(\overline{\Omega})$}}.	
\end{equation}
Therefore by the Duhamel representation to the second equation of \eqref{2.20a}, we obtain
\begin{eqnarray}\label{4.6}
\|\tilde{v}_{\varepsilon}(\cdot,t)\|_{W^{1,q}(\Omega)}&=\left\|e^{(\Delta-D^{-1})}\tilde{v}_{\varepsilon}(\cdot,t-1)+D^{-1}\displaystyle\int^t_{t-1}e^{(t-s)
(\Delta-D^{-1})}u_{\varepsilon}(\cdot, D^{-1}s) ds\right\|_{W^{1,q}(\Omega)}\nonumber\\
&\leq \|e^{\Delta}\tilde{v}_{\varepsilon}(\cdot,t-1)\|_{W^{1,q}(\Omega)}+D^{-1}
\displaystyle \int^{t}_{t-1}
\|e^{(t-s)\Delta}u_{\varepsilon}(\cdot, D^{-1}s)\|_{W^{1,q}(\Omega)}ds.	
\end{eqnarray}
Due to  \eqref{4.4} and  \eqref{4.5}, we have
\begin{equation}\label{4.7}
\|e^{\Delta}\tilde{v}_{\varepsilon}(\cdot,t-1)\|_{W^{1,q}(\Omega)}\leq c_1\|\tilde{v}_{\varepsilon}(\cdot,t-1)\|_{L^1(\Omega)}	
\end{equation}
and
\begin{equation}\label{4.8}
\begin{array}{rl}
&\displaystyle \int^{t}_{t-1}\|e^{-(t-s)\Delta}u_{\varepsilon}(\cdot, D^{-1}s)\|_{W^{1,q}(\Omega)}\\
\leq &  c_2\displaystyle\int^{t}_{t-1}(t-s)^{-\frac{1}{2}-(\frac{1}{p}-\frac{1}{q})}\|u_{\varepsilon}(\cdot,D^{-1}s)\|_{L^p(\Omega)}ds\\
\leq & c_2\displaystyle\left\{\int^{t}_{t-1}(t-s)^{^{-\frac{p}{p-1}(\frac{1}{2}+\frac{1}{p}-\frac{1}{q})}}ds\right\}^{\frac{p-1}{p}}
\left\{\int^{t}_{t-1}\|u_{\varepsilon}(\cdot,D^{-1}s)\|_{L^p(\Omega)}^{p}ds\right\}^{\frac{1}{p}}\\
\leq &  c_2 \displaystyle(\int^{1}_{0}\sigma^{^{-\frac{p}{p-1}\cdot\left(\frac{1}{2}+\frac{1}{p}-\frac{1}{q}\right)}}d\sigma)^{\frac{p-1}{p}}
\left\{\int^{t}_{t-1}\|u_{\varepsilon}(\cdot,D^{-1}s)\|_{L^p(\Omega)}^{p}ds\right\}^{\frac{1}{p}}\\
\leq &  c_2 D^{\frac 1p}
\displaystyle(\int^{1}_{0}\sigma^{^{-\frac{p}{p-1}\cdot\left(\frac{1}{2}+\frac{1}{p}-\frac{1}{q}\right)}}d\sigma)^{\frac{p-1}{p}}
\left\{
\int^{D^{-1}t}_{D^{-1}t-D^{-1}}\|u_{\varepsilon}(\cdot,s)\|_{L^p(\Omega)}^{p}ds
\right\}^{\frac1p}\\
\leq &  c_3 D^{\frac 1p},
\end{array}
\end{equation}
where due to $D\geq 1$ and the application of Lemma \ref{lemma3.3}, we have
\begin{equation*}
\int^{D^{-1}t}_{D^{-1}t-D^{-1}}\|u_{\varepsilon}(\cdot,s)\|_{L^p(\Omega)}^{p}ds
\leq
\int^{D^{-1}t}_{D^{-1}t-1}\|u_{\varepsilon}(\cdot,s)\|_{L^p(\Omega)}^{p}ds \leq c_4
\end{equation*}
 and  the finiteness of
$\int^{1}_{0}\sigma^{^{-\frac{p}{p-1}\cdot\left(\frac{1}{2}+\frac{1}{p}-\frac{1}{q}\right)}}d\sigma$
due to \eqref{4.3}.
 Hence combining \eqref{4.6}  with  \eqref{4.7} and \eqref{4.8} gives
\begin{equation*}\begin{array}{rl}
 \|v_{\varepsilon}(\cdot,t)\|_{W^{1,q}(\Omega)}&
 \leq c_2\|\tilde{v}_{\varepsilon}(\cdot,t-1)\|_{L^1(\Omega)}	
 +c_3 D^{\frac 1p-1}\\
 &\leq \displaystyle c_2(\int_{\Omega}u_0+\beta\int_{\Omega}w_0)+c_3
  \end{array}
 \end{equation*}
 for all $t>2$ and thus completes the proof of \eqref{4.1}.

Next due to $\|w_{\varepsilon}(\cdot,t)\|_{L^{\infty}(\Omega)}\leq \|w_{0}\|_{L^{\infty}(\Omega)}$, an  application of Duhamel representation to the third equation in \eqref{2.5} yields
$$
\|w_{\varepsilon}(\cdot,t)\|_{W^{1,q}(\Omega)}\leq \left\|e^{\Delta}w_{\varepsilon}\left(\cdot,t-1\right)\right\|_{W^{1,q}(\Omega)}\nonumber\\
+f(\|w_{0}\|_{L^{\infty}(\Omega)})\int^{t}_{t-1}\left\|e^{(t-s)\Delta}u_{\varepsilon}(\cdot,s)\right\|_{W^{1,q}(\Omega)}ds,
$$
and thereby \eqref{4.2} can be actually derived as above. % which is used previous to reach the claimed boundedness property of $w_{\varepsilon}$.
\end{proof}

The following lemma will be used in the derivation of regularity features about spatial and temporal derivatives of $u_{\varepsilon}$.
\begin{lemma}\label{lemma4.2}
Let $p>0$ and $\varphi\in C^{\infty}(\overline{\Omega})$, then
\begin{equation}\label{4.9}
\begin{array}{rl}
& \displaystyle \frac{1}{p}\displaystyle \int_{\Omega}\frac{d}{dt}(u_{\varepsilon}+\varepsilon)^p\cdot\varphi +(p-1)M\varepsilon\int_{\Omega}(u_{\varepsilon}+\varepsilon)^{p-2}(u_{\varepsilon}+1)^{M-1}|\nabla u_{\varepsilon}|^2\varphi\\
=&(1-p)\displaystyle \int_{\Omega}(mu_{\varepsilon}+\varepsilon)(u_{\varepsilon}+\varepsilon)^{m+p-4}|\nabla u_{\varepsilon}|^2\varphi
+\alpha(p-1)\displaystyle \int_{\Omega}(u_{\varepsilon}+\varepsilon)^{m+p-3}v_{\varepsilon}^{-\alpha-1}\nabla u_{\varepsilon}\cdot \nabla v_{\varepsilon}\varphi\\
&+(1-p)\displaystyle \int_{\Omega}(mu_{\varepsilon}+\varepsilon)(u_{\varepsilon}+\varepsilon)^{m+p-3}v_{\varepsilon}^{-\alpha}
\nabla u_{\varepsilon}\cdot\nabla \varphi-M\varepsilon\int_{\Omega}(u_{\varepsilon}+\varepsilon)^{p-1}(u_{\varepsilon}+1)^{M-1}\nabla u_{\varepsilon}\cdot \nabla \varphi\\
&+ \alpha\displaystyle \int_{\Omega}u_{\varepsilon}(u_{\varepsilon}+\varepsilon)^{m+p-2}v_{\varepsilon}^{-\alpha-1}\nabla v_{\varepsilon}\cdot\nabla \varphi+\beta \displaystyle \int_{\Omega}u_{\varepsilon}f(w_{\varepsilon})(u_{\varepsilon}+\varepsilon)^{p-1}\varphi
\end{array}
\end{equation}
for all $t>0$ and $\varepsilon\in(0,1)$.
\end{lemma}
\begin{proof}
This can be verified by straightforward computation.	
\end{proof}

Thanks to the boundedness of  $\|\nabla v_{\varepsilon}(\cdot,t)\|_{L^q(\Omega)}$  with some $q>2$ in Lemma \ref{lemma4.1}, we can achieve the following $D$-independent $L^{p}$-estimate of $u_{\varepsilon}$ with finite $p$.
\begin{lemma}\label{Lemma4.3}
Let $m>1$ and  $D\geq 1$. Then for any $p>1$, there exist  $t_0\geq 2$ and constant $C(p)>0$  such that
\begin{equation}\label{4.10}
\|u_{\varepsilon}(\cdot,t)\|_{L^{p}(\Omega)}\leq C(p)	
\end{equation}
for all $t>t_0$ and $\varepsilon\in(0,1)$.
\end{lemma}

\begin{proof}
According to Lemma \ref{Lemma2.3} and Lemma \ref{lemma2.4}, one can find $c_i>0 (i=1,2) $ independent of $D\geq 1$  fulfilling
\begin{equation}\label{4.11}
v_{\varepsilon}^{-\alpha}(x,t) \geq c_1, \quad   v_{\varepsilon}^{-\alpha-2}(x,t)\leq c_2 ~~\hbox{in}~~\Omega\times(2,\infty)
\end{equation}
for all $\varepsilon\in(0,1)$.

Letting $\varphi \equiv 1$ in \eqref{4.9} and by Young's inequality, we have
\begin{equation*}
\begin{array}{rl}
&  \displaystyle\frac{d}{dt} \int_{\Omega}(u_{\varepsilon}+\varepsilon)^p+
p(p-1)\displaystyle \int_{\Omega}(mu_{\varepsilon}+\varepsilon)(u_{\varepsilon}+\varepsilon)^{m+p-4}v_{\varepsilon}^{-\alpha}|\nabla u_{\varepsilon}|^2
+\int_{\Omega}(u_{\varepsilon}+\varepsilon)^{p}\\
%+(p-1)M\varepsilon\int_{\Omega}(u_{\varepsilon}+\varepsilon)^{p-2}(u_{\varepsilon}+1)^{M-1}|\nabla u_{\varepsilon}|^2\\
\leq &\alpha p(p-1)\displaystyle \int_{\Omega}u_{\varepsilon}(u_{\varepsilon}+\varepsilon)^{m+p-3}v_{\varepsilon}^{-\alpha-1}\nabla u_{\varepsilon}\cdot \nabla v_{\varepsilon}+\beta p \displaystyle \int_{\Omega}u_{\varepsilon}f(w_{\varepsilon})(u_{\varepsilon}+\varepsilon)^{p-1}+\int_{\Omega}(u_{\varepsilon}+\varepsilon)^{p}\\
\leq &\displaystyle\frac{p(p-1)}2 \displaystyle \int_{\Omega}(mu_{\varepsilon}+\varepsilon)(u_{\varepsilon}+\varepsilon)^{m+p-4}v_{\varepsilon}^{-\alpha}|\nabla u_{\varepsilon}|^2+
\displaystyle\frac{\alpha^2p(p-1)}2 \displaystyle \int_{\Omega}(u_{\varepsilon}+\varepsilon)^{m+p-1} v_{\varepsilon}^{-\alpha-2}|\nabla v_{\varepsilon}|^2+\\
& +\beta pf (\|w_0\|_{L^{\infty}(\Omega)}) \displaystyle \int_{\Omega}u_{\varepsilon} (u_{\varepsilon}+\varepsilon)^{p-1}+\int_{\Omega}(u_{\varepsilon}+\varepsilon)^{p}
\end{array}
\end{equation*}
Furthermore,  recalling  \eqref{4.11}, we can find $c_3>0$ and $c_4>0$ independent of $p$ such  that
\begin{equation}\label{4.12}
\begin{array}{rl}
&  \displaystyle\frac{d}{dt} \int_{\Omega}(u_{\varepsilon}+\varepsilon)^p+c_3
\displaystyle \int_{\Omega}|\nabla (u_{\varepsilon}+\varepsilon)^{\frac{m+p-1}{2}}|^2
+\int_{\Omega}(u_{\varepsilon}+\varepsilon)^{p}\\
\leq & c_4p^2\displaystyle\int_{\Omega}(u_{\varepsilon}+\varepsilon)^{m+p-1} |\nabla v_{\varepsilon}|^2+ c_4 p \int_{\Omega}(u_{\varepsilon}+\varepsilon)^{p}.
\end{array}
\end{equation}
 According to \eqref{4.1},  $\|\nabla v_{\varepsilon}\|^2_{L^q(\Omega)}\leq c_5$ for any fix $q\in (2,\frac {2(m+1)}{(3-m)_+})$, and hence
the H\"{o}lder inequality yields
\begin{equation}\label{4.13}
\begin{array}{rl}
 &c_4p^2\displaystyle\int_{\Omega}(u_{\varepsilon}+\varepsilon)^{m+p-1} |\nabla v_{\varepsilon}|^2
\\
 \leq
 &
  c_4p^2\left\{
  \displaystyle\int_{\Omega}(u_{\varepsilon}+\varepsilon)^{\frac{(m+p-1)q}{q-2}}
  \right\}
  ^{1-\frac2q}
 \|\nabla v_{\varepsilon}\|^2_{L^q(\Omega)}  \\
 \leq
 &   c_4c_5p^2\|( u_{\varepsilon}+\varepsilon)^{\frac{m+p-1}{2}}\|^2_{L^{\frac{2q}{q-2}}(\Omega)}\\
 \leq & \displaystyle\frac{c_3}4\displaystyle \int_{\Omega}|\nabla (u_{\varepsilon}+\varepsilon)^{\frac{m+p-1}{2}}|^2+c_6(p),
\end{array}
\end{equation}
where  we have used an Ehrling-type inequality due to  $ W^{1,2}(\Omega)\hookrightarrow L^{\frac{2q}{q-2}}(\Omega)$ in two-dimensional setting and \eqref{2.13}.

On the other hand, since
 $$
 c_4 p \int_{\Omega}(u_{\varepsilon}+\varepsilon)^{p}\leq \eta \int_{\Omega} (u_{\varepsilon}+\varepsilon)^{m+p-1}+c(\eta)=
 \eta \|( u_{\varepsilon}+\varepsilon)^{\frac{m+p-1}{2}}\|^2_{L^2(\Omega)}+c(\eta)
 $$
 for any $\eta>0$,
 we also have
\begin{equation}\label{4.14}
c_4p\displaystyle\int_{\Omega}(u_{\varepsilon}+\varepsilon)^{p}
  \leq  \displaystyle\frac{c_3}4\displaystyle \int_{\Omega}|\nabla (u_{\varepsilon}+\varepsilon)^{\frac{m+p-1}{2}}|^2+c_7(p).
\end{equation}
Now inserting \eqref{4.14} and \eqref{4.13} into \eqref{4.12}, we infer that for all $t\geq 2$
\begin{equation}\label{5.10}
\frac{d}{dt}\int_{\Omega}(u_{\varepsilon}+\varepsilon)^{p}+\int_{\Omega}(u_{\varepsilon}+\varepsilon)^p\leq c_8(p)	
\end{equation}
with $c_8(p)>0$ independent of $D\geq 1$, which along with a standard comparison argument implies that there exists $t_0>2$
such that
\begin{equation}\label{5.11}
\int_{\Omega}(u_{\varepsilon}(\cdot,t)+\varepsilon)^{p}\leq 2c_8(p)
\end{equation}
for all $t\geq t_0$
and thus yields the claimed conclusion.
\end{proof}

With  the $L^{p}$-estimate  of $u_{\varepsilon}$ at hand, the standard Moser-type iteration can be immediately applied in our approaches to obtain further regularity concerning $L^{\infty}$-norm of $u_{\varepsilon}$ (see Lemma A.1 of \cite{TaoWinklerJDE} for example) and we list the result here without proof.
\begin{lemma}\label{lemma4.4}
Assume  that $m>1, \alpha>0$  and  $D\geq 1$, then there exists $C>0$ such that
\begin{equation}\label{4.17}
\|u(\cdot,t)\|_{L^{\infty}(\Omega)}\leq C	
\end{equation}
for all $t>t_0$  with $t_0$ as in Lemma \ref{Lemma4.3}.
\end{lemma}
\begin{remark}
It should be mentioned  that when  $m>1, \alpha>0$  and  $D>0$, one can obtain the boundedness of $L^{\infty}$-norm  of $u_{\varepsilon}$
  for all $t>0$ by the above argument (also see \cite{WNon2} for reference). However, the explicit dependence of $\|u_{\varepsilon}(\cdot,t)\|_{L^{p}(\Omega)}$ on $D$  is required to investigate the large time behavior of solutions in the sequel.
 % ~but the bound we obtained relies on the parameter $D$.~
 Hence   $D\geq 1$  is imposed specially for the convenience of our discussion below.
\end{remark}
 At the end of this section, based on the above results  we derive a regularity property for $v$ which
goes beyond those in Lemma \ref{lemma4.1}.

\begin{lemma}\label{lemma4.5}
Let $m>1,\alpha>0 $ and suppose that $D\geq 1$.  Then there exists constant $C>0 $  independent of $D$ and $\varepsilon$ such that
\begin{equation}\label{4.18}
	\|\nabla v_\varepsilon(\cdot,t)\|_{L^{\infty}(\Omega)}\leq C
\end{equation}
as well as
\begin{equation}\label{4.19}
\|\nabla w_\varepsilon(\cdot,t)\|_{L^{\infty}(\Omega)}\leq C
\end{equation}
for  all $t>t_0 +1$.
\end{lemma}
\begin{proof}  Due to
$
\|\nabla e^{\Delta}\tilde{v}(\cdot,\tilde{t}-1)
\|_{L^{\infty}(\Omega)}
\leq c_1\|\tilde{v}(\cdot,\tilde{t}-1)\|_{L^1(\Omega)}	
$,
 as the proof of Lemma \ref{lemma4.1}, we use  Duhamel formula of \eqref{2.20a}  in the following way
 \begin{equation*}
 \begin{array}{rl}
\|\nabla \tilde{v}(\cdot,\tilde{t})\|_{L^{\infty}(\Omega)}&=
\left\|\nabla e^{(\Delta-D^{-1})}\tilde{v}(\cdot,\tilde{t}-1)+D^{-1}\displaystyle\int^{\tilde{t}}_{\tilde{t}-1}\nabla e^{(t-s)
(\Delta-D^{-1})}u(\cdot, D^{-1}s) ds\right\|_{L^{\infty}(\Omega)}\\
&\leq \|\nabla e^{\Delta}\tilde{v}(\cdot,\tilde{t}-1)\|_{L^{\infty}(\Omega)}+D^{-1}
\displaystyle \int^{\tilde{t}}_{\tilde{t}-1}
\|\nabla e^{(\tilde{t}-s)\Delta}u(\cdot, D^{-1}s)\|_{L^{\infty}(\Omega)} d s\nonumber\\
&\leq  c_1\|\tilde{v}(\cdot,\tilde{t}-1)\|_{L^1(\Omega)}+c_2D^{-1}\displaystyle \int^{\tilde{t}}_{\tilde{t}-1}
(1+(\tilde{t}-s)^{-\frac34})ds \max_{\tilde{t}-1\leq s\leq \tilde{t}}\|u(\cdot,D^{-1} s)\|_{L^{4}(\Omega)}.	
\end{array}
\end{equation*}
for all $\tilde{t}>t_0 D+1$, which along with \eqref{4.17} readily leads to \eqref{4.18}. It is obvious that \eqref{4.19} can be proved similarly.
\end{proof}

%\section{First-order regularity properties of $u_{\varepsilon}$}
%In view of nonlinear manner in which $u_{\varepsilon}$ appears in (\ref{2.9}),~it seems
%some more further $\varepsilon$-independent regularity information capable of implying at
%least some pointwise convergence properties of $(u_{\varepsilon})_{\varepsilon \in(0,1)}$ along subsequence is required.~This will be the objective of our next lemmata to prepare an argument based on an application of Aubin-Lions Lemma.

\section{Asymptotic behavior}
\subsection{Weak decay information}
The standard parabolic regularity property becomes applicable  to  improve the regularity of $u,v$ and $w$ as follows.

\begin{lemma}\label{lemma5.1}
Let $(u,v,w)$ be the nonnegative global solution of \eqref{1.10}--\eqref{1.12} obtained in Lemma \ref{Lemma2.1}. Then there exist $\kappa\in(0,1)$ and  $C>0$ such that for all $t>t_0$
\begin{equation}\label{5.1}
\|u\|_{C^{\kappa,\frac{\kappa}{2}}(\overline{\Omega}\times[t,t+1])}\leq C
\end{equation}
as well as
\begin{equation}\label{5.2}
\|v\|_{C^{2+\kappa,1+\frac{\kappa}{2}}(\overline{\Omega}\times[t,t+1])}+
\|w\|_{C^{2+\kappa,1+\frac{\kappa}{2}}(\overline{\Omega}\times[t,t+1])}\leq C.
\end{equation}
\end{lemma}
\begin{proof}
We rewrite the first equation of \eqref{2.5} in the form
$$
u_{\varepsilon t}=\nabla \cdot a(x,t,u_{\varepsilon},\nabla u_{\varepsilon }) +b(x,t,u_{\varepsilon },\nabla u_{\varepsilon })
$$
where
 $$ a(x,t,u_{\varepsilon},\nabla u_{\varepsilon })=(\varepsilon  M (u_{\varepsilon }+1)^{M-1}+ mu_{\varepsilon }^{m-1}v_{\varepsilon }^{-\alpha}) \nabla u_{\varepsilon }-\alpha u_{\varepsilon }^m v_{\varepsilon }^{-\alpha-1}\nabla v_{\varepsilon }$$
 and
 $$b(x,t,u_{\varepsilon },\nabla u_{\varepsilon })=\beta u_{\varepsilon }f(w_{\varepsilon }).$$
According to Lemmas \ref{Lemma2.3},  \ref{lemma2.4},  \ref{lemma4.4} and
 \ref{lemma4.5},  there exist two constants  $c_1>0  $ and $c_2>0 $ independent of $D\geq 1$  satisfying
\begin{equation}
c_1 \leq v_{\varepsilon }^{-\alpha}(x,t) \leq c_2 ~~\hbox{in}~~\Omega\times(t_0,\infty) \end{equation}
and $$   \|v_{\varepsilon }^{-\alpha-1}(\cdot,t)\|_{L^\infty(\Omega)}+  \|\nabla v_{\varepsilon }(\cdot,t)\|_{L^\infty(\Omega)}+ \|u_{\varepsilon }(\cdot,t)\|_{L^\infty(\Omega)}
+ \|w_{\varepsilon }(\cdot,t)\|_{L^\infty(\Omega)}\leq c_2
 ~~\hbox{for}~~t\geq t_0. $$
 This guarantees that for all  $ (x,t)\in \Omega\times (t_0,\infty)$
 $$
a(x,t,u_{\varepsilon},\nabla u_{\varepsilon })\cdot \nabla u_{\varepsilon}\geq  \frac {c_1m}2 u_{\varepsilon}^{m-1}|\nabla u_{\varepsilon}|^2-c_3,
 $$
$$
|a(x,t,u_{\varepsilon},\nabla u_{\varepsilon })|\leq mc_4 u_{\varepsilon }^{m-1}|\nabla u_{\varepsilon }|+c_4|u_{\varepsilon }|^{\frac{m-1}2}
 $$
 and $$
 |b(x,t,u_{\varepsilon },\nabla u_{\varepsilon })|\leq c_5
$$ with some  constants $c_i>0$ $(i=3,4,5)$ independent of $\varepsilon>0$.
Therefore as an application of the known result on H\"{o}lder regularity in scalar
parabolic equations (\cite{PV}),  there exist $\kappa_1\in(0,1)$ and  $C>0$ such that for all $t>t_0$ and $ \varepsilon\in (0,1)$,
$$
\|u_{\varepsilon }\|_{C^{\kappa_1,\frac{\kappa_1}{2}}(\overline{\Omega}\times[t,t+1])}\leq C,
$$
which along with \eqref{2.7} readily entails \eqref{5.1} with $\kappa=\kappa_1$. Similarly one can also conclude that there exist $\kappa_2\in(0,1)$ and  $C>0$ such that
\begin{equation}\label{5.4}
\|v\|_{C^{\kappa_2,\frac{\kappa_2}{2}}(\overline{\Omega}\times[t,t+1])}+
\|w\|_{C^{\kappa_2,\frac{\kappa_2}{2}}(\overline{\Omega}\times[t,t+1])}\leq C.
\end{equation}
Moreover, since $f\in C^1[0, \infty)$, we have
$$\|uf(w)\|_{C^{\kappa_3,\frac{\kappa_3}{2}}(\overline{\Omega}\times[t,t+1])}\leq C
$$
 with $\kappa_3=\min\{\kappa_1,\kappa_2\}$. Thereupon \eqref{5.2}  with $\kappa=\kappa_3$ follows from the standard parabolic Schauder theory (\cite{LSU}).
\end{proof}

The core of our proof of the stabilization result in Theorem 1.1 consists in the following observation.

\begin{lemma}\label{lemma5.2} Assume that
$m>1$ and $D\geq 1$, we have %ceThere exists $C>0$ such that  the weak solution $(u,v,w)$ of \eqref{1.4}  satisfies
\begin{equation}\label{5.5}
\int^{\infty}_0\int_{\Omega}uf(w)< \infty %\quad\textnormal{\emph{for all $\varepsilon\in(0,\varepsilon^{\star})$}}	
\end{equation}
and
\begin{equation}\label{5.6}
\int^{\infty}_0\int_{\Omega}|\nabla w|^2< \infty. %\quad\textnormal{\emph{for all $\varepsilon\in(0,\varepsilon^{\star})$}}.
\end{equation}	
\end{lemma}
\begin{proof}
An integration of the third equation in (\ref{2.10}) yields
\begin{equation*}
\int_{\Omega}w_{\varepsilon}(\cdot,t)+\int^t_0\int_{\Omega}u_{\varepsilon}f(w_{\varepsilon})=\int_{\Omega}w_0
\quad \hbox{for all}~~t>0.	
\end{equation*}
Since $w_{\varepsilon}\geq 0$, this entails
\begin{equation}\label{5.7}
\int^{\infty}_0\int_{\Omega}u_{\varepsilon}f(w_{\varepsilon})
\leq \int_{\Omega}w_0
\end{equation}
which implies \eqref{5.5} on an application of Fatou's lemma, because  $u_{\varepsilon}f(w_{\varepsilon})\rightarrow
uf(w)$ a.e. in $\Omega\times (0,\infty)$.

We test the same equation by $w_{\varepsilon}$ to see that
\begin{equation*}
\frac{1}{2}\int_{\Omega}w^2_{\varepsilon}(\cdot,t)+\int^{t}_0\int_{\Omega}|\nabla w_{\varepsilon}|^2=\frac{1}{2}\int_{\Omega}w^2_0-\int^t_0\int_{\Omega}u_{\varepsilon}f(w_{\varepsilon})w_{\varepsilon}
\leq \frac{1}{2}\int_{\Omega}w^2_0%\quad \textnormal{\emph{for all $\varepsilon\in(0,\varepsilon^{\star})$}}.
\end{equation*}
and thereby verifies \eqref{5.6} via (2.12).  %the  Fatou lemma. %lower semi-continuity of the norm in $L^2((0,\infty);W^{1,2}(\Omega))$.	
\end{proof}
The above decay information of $w_{\varepsilon}$ seems to be  weak for the derivation of the large-time behavior of $u_{\varepsilon}$ and $v_{\varepsilon}$.  Indeed,  under additional constraint on $D$,  we  obtain the decay information concerning the gradient of $u_{\varepsilon}$ and $v_{\varepsilon}$ which makes our latter analysis possible.
\begin{lemma}\label{lemma5.3}
Let $m>1$ and $\alpha>0$. There exists  $D_0\geq 1$  such that whenever $D>D_0$, the solution of \eqref{1.10}--\eqref{1.12}  constructed in Lemma
\ref{Lemma2.1} satisfies %$C>0$ %and $\varepsilon^{\star}\in(0,1)$
\begin{equation}\label{5.8}
\int^{\infty}_{3}\int_{\Omega}|\nabla u^{\frac{m+1}{2}}%_{\varepsilon}
|^2
< \infty %\quad\textnormal{\emph{for all $\varepsilon\in(0,\varepsilon^{\star})$}}
\end{equation}
as well as
\begin{equation}\label{5.9}
\int^{\infty}_{3}\int_{\Omega}|\nabla v%_{\varepsilon}
|^2
< \infty%\quad\textnormal{\emph{for all $\varepsilon\in(0,\varepsilon^{\star})$}}	
\end{equation}
	\end{lemma}

\begin{proof}
Testing the first equation of \eqref{2.5} by $(u_{\varepsilon}+\varepsilon)$ and applying Young's inequality, we obtain
\begin{eqnarray}\label{6.8}
\frac{d}{dt}\int_{\Omega}(u_{\varepsilon}+\varepsilon)^2+\int_{\Omega}v_{\varepsilon}^{-\alpha}(u_{\varepsilon}+\varepsilon)^{m-1}|\nabla u_{\varepsilon}|^2\nonumber\\
\leq \alpha^2 \int_{\Omega}v_{\varepsilon}^{-\alpha-2} (u_{\varepsilon}+\varepsilon)^{m+1}|\nabla v_{\varepsilon}|^2+ 2\int_{\Omega}(u_{\varepsilon}+\varepsilon)u_{\varepsilon} f(w_{\varepsilon}).	
\end{eqnarray}
On the other hand, let $\mu_{\varepsilon}(t)=\left(\frac{1}{|\Omega|}\int_{\Omega}u^{\frac{m+1}{2}}_{\varepsilon}(\cdot,t)\right)^{\frac{2}{m+1}}$, then testing the second equation of \eqref{2.5} by $-\Delta v_{\varepsilon}$ shows
\begin{eqnarray}\label{6.9}
\frac{d}{dt}\int_{\Omega}|\nabla v_{\varepsilon}|^2+D\int_{\Omega}(\Delta v_{\varepsilon})^2+2\int_{\Omega}|\nabla v_{\varepsilon}|^2\leq \frac{1}{D}\int_{\Omega}|u_{\varepsilon}(\cdot,t)-\mu_{\varepsilon}(t)|^2. 	
\end{eqnarray}
Hence combining \eqref{6.8} and \eqref{6.9}, we obtain that for any  $\eta>0$
\begin{equation}\label{6.10}
\begin{array}{rl}
&\displaystyle\frac{d}{dt}\left(\int_{\Omega}(u_{\varepsilon}+\varepsilon)^2+\eta\int_{\Omega}|\nabla v_{\varepsilon}|^2\right)+\eta D\int_{\Omega}|\Delta v_{\varepsilon}|^2+2\eta \int_{\Omega}|\nabla v_{\varepsilon}|^2
+\int_{\Omega}v^{-\alpha}_{\varepsilon}(u_{\varepsilon}+\varepsilon)^{m-1}|\nabla u_{\varepsilon}|^2\\[3mm]
\leq
&
\displaystyle\frac{\eta}{D}\int_{\Omega}|u_{\varepsilon}(\cdot,t)-\mu_{\varepsilon}(t)|^2
+\alpha^2 \int_{\Omega}v_{\varepsilon}^{-\alpha-2} (u_{\varepsilon}+\varepsilon)^{m+1}|\nabla v_{\varepsilon}|^2
+2\int_{\Omega}(u_{\varepsilon}+\varepsilon)u_{\varepsilon}f(w_{\varepsilon}).
\end{array}
\end{equation}
In view of Lemma \ref{Lemma2.3} and Lemma \ref{lemma2.4}, there exist $c_i>0 (i=1,2) $ independent of $D\geq 1$  satisfying
\begin{equation}
v_{\varepsilon}^{-\alpha}(x,t) \geq c_1, \quad   v_{\varepsilon}^{-\alpha-2}(x,t)\leq c_2 ~~\hbox{in}~~\Omega\times(2,\infty)
\end{equation}
for all $\varepsilon\in(0,1)$. Therefore from \eqref{6.10}, it follows that
\begin{equation}\label{6.12}
\begin{array}{rl}
&\displaystyle\frac{d}{dt}\left(\int_{\Omega}(u_{\varepsilon}+\varepsilon)^2+\eta\int_{\Omega}|\nabla v_{\varepsilon}|^2\right)+\eta D\int_{\Omega}|\Delta v_{\varepsilon}|^2+2\eta \int_{\Omega}|\nabla v_{\varepsilon}|^2
+c_1\int_{\Omega}(u_{\varepsilon}+\varepsilon)^{m-1}|\nabla u_{\varepsilon}|^2\\[3mm]
\leq
&
\displaystyle\frac{\eta}{D}\int_{\Omega}|u_{\varepsilon}(\cdot,t)-\mu_{\varepsilon}(t)|^2
+\alpha^2 c_2 \int_{\Omega} (u_{\varepsilon}+\varepsilon)^{m+1}|\nabla v_{\varepsilon}|^2
+2\int_{\Omega}(u_{\varepsilon}+\varepsilon)u_{\varepsilon}f(w_{\varepsilon}).
\end{array}
\end{equation}
According to Lemma \ref{Lemma4.3} with $p=2(m+1)$, we have
$$
\displaystyle\left(\int_{\Omega}(u_{\varepsilon}+\varepsilon)^{2(m+1)}\right)^{\frac 12}
\leq c_3,
$$
and then use  the Gagliardo--Nirenberg inequality and H\"older inequality to arrive at
\begin{equation}\label{6.13}
\begin{array}{rl}
&\displaystyle\int_{\Omega}(u_{\varepsilon}+\varepsilon)^{m+1}|\nabla v_{\varepsilon}|^2\\
 \leq &\displaystyle\left(\int_{\Omega}(u_{\varepsilon}+\varepsilon)^{2(m+1)}\right)^{\frac{1}{2}}\left(\int_{\Omega}|\nabla v_{\varepsilon}|^4\right)^\frac{1}{2}\\
\leq & c_4\displaystyle\left(\int_{\Omega}(u_{\varepsilon}+\varepsilon)^{2(m+1)}\right)^{\frac{1}{2}}\left(\|\Delta v_{\varepsilon}\|^2
+\|\nabla v_{\varepsilon}\|^2_{L^2(\Omega)}\right)	
\\
\leq & c_3c_4 (\|\Delta v_{\varepsilon}\|^2
+\|\nabla v_{\varepsilon}\|^2_{L^2(\Omega)}).
\end{array}
\end{equation}
Therefore inserting \eqref {6.13} into \eqref {6.12} yields
\begin{equation}\label{6.14}
\begin{array}{rl}
&\displaystyle\frac{d}{dt}(\int_{\Omega}(u_{\varepsilon}+\varepsilon)^2+\eta\int_{\Omega}|\nabla v_{\varepsilon}|^2)+\eta D\int_{\Omega}|\Delta v_{\varepsilon}|^2+2\eta \int_{\Omega}|\nabla v_{\varepsilon}|^2
+\frac{4c_1}{(m+1)^2}\int_{\Omega} |\nabla (u_{\varepsilon}+\varepsilon)^{\frac{m+1}{2}}|^2\\[3mm]
\leq
&
\displaystyle\frac{\eta}{D}\int_{\Omega}|u_{\varepsilon}(\cdot,t)-\mu_{\varepsilon}(t)|^2
+\alpha^2 c_2 c_3c_4 (\|\Delta v_{\varepsilon}\|^2_{L^2(\Omega)}
+\|\nabla v_{\varepsilon}\|^2_{L^2(\Omega)})
+2\int_{\Omega}(u_{\varepsilon}+\varepsilon)u_{\varepsilon}f(w_{\varepsilon}).
%2(\|u_{\varepsilon}(\cdot,t)\|_{L^\infty(\Omega)}+\varepsilon)\int_{\Omega}u_{\varepsilon}f(w_{\varepsilon}).
\end{array}
\end{equation}

By the elementary inequality:
 $$
\frac{\xi^\mu-\eta^\mu}{\xi-\eta}
\geq \eta^{\mu-1}~~\hbox{for}~~\mu\ge 1,\xi\geq 0, \eta\geq 0~~\hbox{ and}~~\xi\neq\eta,
$$
we have
 $$
|u_{\varepsilon}^{\frac{m+1}2}(\cdot,t)-\mu_{\varepsilon}^{\frac{m+1}2}|
\geq \mu_{\varepsilon}(\cdot,t)^{\frac{m-1}2}|u_{\varepsilon}(\cdot,t)-\mu_{\varepsilon}(t)|
$$
and thus
\begin{equation}\label{6.15}
{\mu}^{m-1}_{\varepsilon}(t)\int_{\Omega}|u_{\varepsilon}(\cdot,t)-\mu_{\varepsilon}(t)|^2\leq \int_{\Omega}|u^{\frac{m+1}{2}}_{\varepsilon}(\cdot,t)-\mu^{\frac{m+1}{2}}_{\varepsilon}(t)|^2.
\end{equation}
Furthermore by the H\"older inequality and the noncreasing property of $t\mapsto \int_\Omega u_\varepsilon(\cdot,t)$,
$$\mu_\varepsilon(t)\geq \frac 1{|\Omega|}\int_\Omega u_\varepsilon(\cdot,t)\geq \frac 1{|\Omega|}\int_\Omega u_0 $$
 and thereby the
Poincar\'e inequality entails that for some $c_5>0$
\begin{equation}\label{5.18}
\begin{array}{rl}
&\overline{u_0}^{m-1}\displaystyle\int_{\Omega}|u_{\varepsilon}(\cdot,t)-\mu_{\varepsilon}(t)|^2\\[3mm]
\leq &\displaystyle\int_{\Omega}|u^{\frac{m+1}{2}}_{\varepsilon}(\cdot,t)-\mu^{\frac{m+1}{2}}_{\varepsilon}(t)|^2\\[2mm]
\leq  & c_5\displaystyle\int_{\Omega}|\nabla u^{\frac{m+1}{2}}_{\varepsilon}|^2\\[2mm]
\leq  & c_5\displaystyle\int_{\Omega}|\nabla (u_{\varepsilon}+\varepsilon)^{\frac{m+1}{2}}|^2.	
\end{array}
\end{equation}

Hence substituting \eqref{5.18} into \eqref{6.14} shows that
\begin{equation*}
\begin{array}{rl}
&\displaystyle\frac{d}{dt}(\int_{\Omega}(u_{\varepsilon}+\varepsilon)^2+\eta\int_{\Omega}|\nabla v_{\varepsilon}|^2)+
(\frac{4c_1}{(m+1)^2}-\frac {\eta c_5}{D\overline{u_0}^{m-1}})\int_{\Omega} |\nabla (u_{\varepsilon}+\varepsilon)^{\frac{m+1}{2}}|^2\\[3mm]
\leq
&
\displaystyle(\alpha^2 c_2 c_3c_4-\eta D) \|\Delta v_{\varepsilon}\|^2_{L^2(\Omega)}
+(\alpha^2 c_2 c_3c_4-2\eta)\|\nabla v_{\varepsilon}\|^2_{L^2(\Omega)}
+2\displaystyle\int_{\Omega}(u_{\varepsilon}+\varepsilon)u_{\varepsilon}f(w_{\varepsilon})\\
\leq & \displaystyle(\alpha^2 c_2 c_3c_4-\eta) \|\Delta v_{\varepsilon}\|^2_{L^2(\Omega)}
+(\alpha^2 c_2 c_3c_4-2\eta)\|\nabla v_{\varepsilon}\|^2_{L^2(\Omega)}
+2\|u_{\varepsilon}(\cdot,t)\|_{L^\infty(\Omega)}+1)\displaystyle\int_{\Omega}u_{\varepsilon}f(w_{\varepsilon})
\end{array}
\end{equation*}
and  hence completes the proof upon the choice of $D_0:=\max\{1,\frac {\alpha^2 c_2 c_3c_4c_5(m+1)^2}{3c_1\overline{u_0}^{m-1}}\}$. Indeed, for any
$D>D_0$, it is possible to find $\eta>0$ such that
 $$\frac{3c_1}{(m+1)^2}\geq \frac {\eta c_5}{D\overline{u_0}^{m-1}},~~
\alpha^2 c_2 c_3c_4\leq \eta
$$
and thereby
 \begin{equation*}
\begin{array}{rl}
&\displaystyle\frac{d}{dt}(\int_{\Omega}|u_{\varepsilon}+\varepsilon|^2+\eta\int_{\Omega}|\nabla v_{\varepsilon}|^2)+
\frac{c_1}{(m+1)^2}\int_{\Omega} |\nabla (u_{\varepsilon}+\varepsilon)^{\frac{m+1}{2}}|^2
+\eta\int_{\Omega}|\nabla v_{\varepsilon}|^2\\[3mm]
\leq
&
2(\|u_{\varepsilon}(\cdot,t)\|_{L^\infty(\Omega)}+1)\displaystyle\int_{\Omega}u_{\varepsilon}f(w_{\varepsilon}).
\end{array}
\end{equation*}
 Therefore, in view of \eqref{5.7}, \eqref{4.17} and \eqref{4.18}  we see that for any $t>3$,
 \begin{equation}\label{5.19}
\displaystyle\int^t_3\int_{\Omega} |\nabla (u_{\varepsilon}+\varepsilon)^{\frac{m+1}{2}}|^2
+\displaystyle\int^t_3\int_{\Omega}|\nabla v_{\varepsilon}|^2
\leq c_6+
c_6\displaystyle\int^\infty_3\displaystyle\int_{\Omega}u_{\varepsilon}f(w_{\varepsilon})\leq c_6+c_6\int_\Omega w_0.
\end{equation}
with constant $c_6>0$ independent of $\varepsilon$ and time $t$,
  which  implies that \eqref{5.8} and \eqref{5.9} is valid due to the lower semicontinuity of  norms.
\end{proof}

\subsection{Decay of $w$}
The integrability statement in Lemma \ref{lemma5.2} can  be turned into
the decay property  of $w$ with respect
to the norm  in $L^\infty(\Omega)$, thanks to the fact that
  $\|u(\cdot,t)\|_{L^1(\Omega)}$  is increasing with time, while $\|w(\cdot,t)\|_{L^\infty(\Omega)}$ is nonincreasing.
 %we can claim that the third component of (\ref{MainSystem}) indeed decays thanks to the H\"older estimates from Lemma \ref{Lemma 8.1} and Lemma \ref{8.2}.
\begin{lemma}\label{lemma5.4}
The third component of  the weak solution of \eqref{1.10}--\eqref{1.12} constructed  in Lemma
\ref{Lemma2.1}
 fulfills
% Let $m>1$ and assume (\ref{VassumptionBase1}) (\ref{VassumptionBase2}) and (\ref{Vassumption}) hold with $\alpha\geq 0$.~Then we have
\begin{equation}\label{5.20}
\|w(\cdot,t)\|_{L^{\infty}(\Omega)}\rightarrow 0 \quad \textnormal{\emph{as $t\rightarrow \infty$}}.	
\end{equation}	
\end{lemma}
\begin{proof} Writing $\overline{u_0}:=\frac 1{|\Omega|}\int_{\Omega}u_0$ and $\overline{f(w)}:=\frac 1{|\Omega|}\int_{\Omega}f(w)$,
we use  the Cauchy--Schwarz inequality and Poincar\'e inequality to see that for all $t>0$
\begin{eqnarray*}
\overline{u_0}\cdot\int_{\Omega}f(w)&= &\int_{\Omega}u\overline {f(w)}\\
&= &\int_{\Omega}u f(w)-\int_{\Omega}u(f(w)-\overline {f(w)})\\
&\leq &\int_{\Omega}u f(w)+c_1\|u\|_{L^{\infty}(\Omega)} \|f'(w)\|_{L^{\infty}(\Omega)}\left\{
\int_{\Omega}|\nabla w|^2\right\}^{\frac{1}{2}}.	
\end{eqnarray*}
Thanks to the boundedness of $u$ and $w$, we have
% indicates $\|F^{\prime}(w_{\varepsilon})\|^2_{L^{\infty}(\Omega)}<\infty$,~thus
\begin{eqnarray*}
{\overline{u_0}}^2\cdot\left\{\int_{\Omega}f(w)\right\}^{2}&\leq &2\left\{\int_{\Omega}u f(w)\right\}^2+c_2\int_{\Omega}|\nabla w|^2\\
&\leq &c_3\int_{\Omega}u f(w)+c_2\int_{\Omega}|\nabla w|^2.
\end{eqnarray*}
 Hence from  Lemma \ref{lemma5.2} it follows that
\begin{equation*}
\int^{\infty}_1\|f(w(\cdot,t))\|^2_{L^1(\Omega)}dt< \infty.
\end{equation*}
which, along with the uniformly H\"older estimate from Lemma \ref{lemma5.1},  implies that
\begin{equation}\label{5.21}
f(w(\cdot,t))\rightarrow 0	\quad \textnormal{in $L^1(\Omega)$}\quad \textnormal{as $t\rightarrow \infty$}
\end{equation}
and thereby we may extract a subsequence
$(t_j)_{j\in \mathbb{N}}
\subset\mathbb{N} $ such that as $t_j\rightarrow\infty$,
 $f(w(\cdot,t_j))\rightarrow 0	$
almost everywhere in  $\Omega $.
Recalling function $f$ is positive on $(0,\infty)$ and $f(0)=0$, this necessarily requires that
$w(\cdot,t_j)\rightarrow 0	$
almost everywhere in  $\Omega $ as $t_j\rightarrow\infty.$
Furthermore, the dominated convergence theorem ensures
that \begin{equation}\label{5.22}
w(\cdot,t_j)\rightarrow 0	\quad \textnormal{in $L^1(\Omega)$}\quad \textnormal{as $t_j\rightarrow \infty$}.
\end{equation}
Now  invoking  the Gagliardo--Nirenberg inequality in two dimensional setting, we have
\begin{equation*}
\|w(\cdot,t_j)\|_{L^{\infty}(\Omega)}\leq c_4\|\nabla w(\cdot,t_j)\|^{\frac{4}{5}}_{L^{4}(\Omega)}\|w(\cdot,t_j)\|^{\frac{1}{5}}_{L^{1}(\Omega)}+c_4\|w(\cdot,t_j)\|_{L^1(\Omega)}	
\end{equation*}
and thus
\begin{equation}\label{5.23}
\|w(\cdot,t_j)\|_{L^{\infty}(\Omega)}\rightarrow 0	\quad \textnormal{as $t_j\rightarrow \infty$}.
\end{equation}
Since $t\mapsto \|w(\cdot,t)\|_{L^{\infty}(\Omega)}$ is noncreasing by Lemma \ref{Lemma2.3}, \eqref{5.20} indeed results
from  \eqref{5.23}.
\end{proof}

\subsection{Convergence of $u$}
In this subsection, we will show that  $u$  stabilizes toward the constant  $\overline{u_0}+\beta \overline{w_0}$ as $t\rightarrow \infty$.
Note that  a first step in this direction  is provided by the finiteness of $\int^{\infty}_{3}\int_{\Omega}|\nabla u^{\frac{m+1}{2}}|^2$
in Lemma \ref{lemma5.3}, which implies that $\|\nabla u^{\frac{m+1}2}(\cdot,t_k)\|_{L^2(\Omega)}$ along a suitable sequence of numbers
$t_k\rightarrow\infty$. However, in order to make sure convergence along the entire net  $t\rightarrow\infty$,  a certain decay property of $u_t$ seems to be required.
\begin{lemma}\label{lemma5.5}
 We have \begin{equation}\label{5.24}
 \int^\infty_3\|
 u_t(\cdot,t)
 \|^2_{(W_0^{1,2}(\Omega))^*}dt
 <\infty.
\end{equation}
\end{lemma}
\begin{proof} For any $\varphi\in  C_0^{\infty}(\Omega)$, multiplying the first
equation in \eqref{2.5} by $\varphi$ and integrating by parts over $\Omega$ yields
\begin{equation*}
\begin{array}{rl}
|\displaystyle\int_\Omega  u_{\varepsilon t}\varphi| &= \left|\displaystyle\int_\Omega
\varepsilon\nabla(u_{\varepsilon}+1)^{M}\cdot\nabla\varphi
+\nabla(u_{\varepsilon}(u_{\varepsilon}+\varepsilon )^{m-1}v_{\varepsilon}^{-\alpha})\cdot\nabla\varphi+
\beta u_{\varepsilon}f(w_{\varepsilon})
%\nabla(\frac{u^m}{v^\alpha})\cdot\nabla\varphi+\beta\int_{\Omega}uf(w)
\varphi\right|\\[3mm]
&\leq \displaystyle\int_\Omega ( M(u_{\varepsilon}+1)^{M-1}|\nabla u_{\varepsilon}|+
mv_{\varepsilon}^{-\alpha}(u_{\varepsilon}+\varepsilon )^{m-1}|\nabla u_{\varepsilon}|+\alpha (u_{\varepsilon}+1)^m v_{\varepsilon}^{-\alpha-1}|\nabla v_{\varepsilon}|)|\nabla \varphi|\\[3mm]
&+
 \displaystyle\beta\int_{\Omega}|u_{\varepsilon}f(w_{\varepsilon})|\|\varphi\|_{L^\infty(\Omega)}\\
&\leq c_1 (\left\{\displaystyle \int_{\Omega} |\nabla ( u_\varepsilon+\varepsilon)^{\frac{m+1}{2}}|^2\right\}^{\frac 12}+\displaystyle\left\{ \int_{\Omega} |\nabla v_\varepsilon|^2\right\}^{\frac 12})\|\varphi\|_{W^{1,2}(\Omega)}+ \beta\displaystyle\int_{\Omega}u_\varepsilon f(w_\varepsilon)\|\varphi\|_{L^\infty(\Omega)}
\end{array}
\end{equation*}
with $c_1>0$ independent of $\varphi$ and $\varepsilon$, where we have used the boundedness of $u_\varepsilon$ and $v_\varepsilon$.

As in the considered two-dimensional setting we have $ W^{1,2}(\Omega)\hookrightarrow L^\infty(\Omega)$,
the above inequality implies that
$$
\|
 u_{\varepsilon t}(\cdot,t)
 \|_{(W_0^{1,2}(\Omega))^*}\leq c_1 (
 \left\{\displaystyle \int_{\Omega} |\nabla (u_\varepsilon+\varepsilon)^{\frac{m+1}{2}}|^2\right\}^{\frac 12}+\displaystyle\left\{ \int_{\Omega} |\nabla v_\varepsilon|^2\right\}^{\frac 12}) +\beta \displaystyle\int_{\Omega}u_\varepsilon f(w_\varepsilon)
$$
for all $t> 3$  and hence for all $ T> 4$,
\begin{equation*}
\displaystyle\int^T_3\|
 u_{\varepsilon t}(\cdot,t)
 \|^2_{(W^{1,2}(\Omega))^*}dt\leq  c_2\left(\int^T_3\displaystyle \int_{\Omega} |\nabla (u_\varepsilon+\varepsilon)^{\frac{m+1}{2}}|^2
 +
 \displaystyle \int^T_3\int_{\Omega} |\nabla v_\varepsilon|^2+
\displaystyle\int^T_3\int_{\Omega}u_\varepsilon f(w_\varepsilon)\right)
\end{equation*}
which together with \eqref{5.19} leads to
$$
\int^\infty_3\|
 u_{\varepsilon t}(\cdot,t)
 \|^2_{(W_0^{1,2}(\Omega))^*}dt
 \leq c_3
$$
with $c_3>0$ independent of $\varepsilon$. Hence  \eqref{5.24} results from  lower semi-continuity of the norm in the Hilbert space 
$L^2((3,\infty);(W_0^{1,2}(\Omega))^*) $ with respect to weak convergence.
\end{proof}
 %, leads to % due to the boundedness of $u$ and $w$  by Lemma \ref{lemma4.4}.Hence
 %the convergence of the integral in.
%results from Lemma \ref{lemma5.3}.
Thanks to above estimates,  we adapt the argument in \cite{Warma} to show that $u$ actually stabilizes  toward  $\overline{u_0}+\beta \overline{w_0}$ in the claimed sense beyond in the weak-$*$ sense in  $L^\infty(\Omega)$.
\begin{lemma}\label{lemma5.6}
Let $m>1,\alpha>0$ and suppose that  $D\geq D_0$. Then we have
\begin{equation}\label{5.26}
\|u(\cdot,t)-u_{\star}\|_{L^{\infty}(\Omega)}\rightarrow 0\quad \textnormal{\emph{as}}~~t\rightarrow \infty,
\end{equation}
where $u_{\star}=\frac{1}{|\Omega|}\int_{\Omega} u_0+\frac{\beta}{|\Omega|}\int_{\Omega}w_0$.	
\end{lemma}
\begin{proof}According to Lemma \ref{lemma5.3} and Lemma \ref{lemma5.5}, one can conclude that
\begin{equation}\label{5.27}
u(\cdot,t)\stackrel{\mathrm{w}^*}{\rightharpoonup}  u_{\star}\quad \hbox {in}~ L^\infty(\Omega)~~\textnormal{\emph{as}}~~t\rightarrow \infty
\end{equation}
In fact, if this conclusion does not hold, then
one can find a
sequence $(t_k)_{k\in \mathbb{N}} \subset (0, \infty)$ such that $ t_k \rightarrow\infty$ as $k \rightarrow\infty$, and   some
$\tilde{\psi}\in  L^1(\Omega) $ such that
$$
\int_\Omega u(x, t_k)\tilde{\psi}dx-\int_\Omega u_{\star}\tilde{\psi}dx\geq c_1
~\hbox{for all}~k\in\mathbb{N} $$
with some $c_1 > 0$. Furthermore, by the boundedness of $u$ and  the density of $C_0^\infty(\Omega )$ in $L^1(\Omega)$, we can choose
$\psi\in  C_0^\infty(\Omega )$ closing $\tilde{\psi}$ in  $L^1(\Omega)$ enough that
\begin{equation*}
\int_\Omega u(x, t_k)\psi dx-\int_\Omega u_{\star}\psi dx\geq \frac{3c_1}4
~~~\hbox{for all}~k\in\mathbb{N}.
\end{equation*}
and then
\begin{equation}\label{5.28}
\displaystyle \int^{t_k+1}_{t_k}\int_\Omega u(x, t)\psi dxdt-\displaystyle \int^{t_k+1}_{t_k}\int_\Omega u_{\star}\psi dxdt\geq \frac{c_1}2
~~~\hbox{for all sufficently large }~k\in\mathbb{N},
 \end{equation}
where we have used the fact that
 \begin{equation*}
 \begin{array}{rl}
&\left|\displaystyle \int^{t_k+1}_{t_k}\int_\Omega( u(x, t)-u(x, t_k))\psi dx\right|
\\[2mm]
=&\left|\displaystyle \int^{t_k+1}_{t_k}\int^t_{t_k}\langle u_t(\cdot,s), \psi(\cdot)\rangle ds dt\right|\\
\leq & \displaystyle \int^{t_k+1}_{t_k}\int^t_{t_k}\|u_t(\cdot,s)\|_{(W^{1,2}_0(\Omega))^*}
ds dt\cdot\|\psi\|_{W^{1,2}_0(\Omega)} \\[4mm]
\leq & \displaystyle \int^{t_k+1}_{t_k}\left\{\int^t_{t_k}\|u_t(\cdot,s)\|^2_{(W^{1,2}_0(\Omega))^*}ds\right\}^{\frac12}|t-t_k|^{\frac12}
 dt\cdot\|\psi\|_{W^{1,2}_0(\Omega)} \\[4mm]
 \leq & \left\{\displaystyle \int^{t_k+1}_{t_k}\int^t_{t_k}\|u_t(\cdot,s)\|^2_{(W^{1,2}_0(\Omega))^*}ds
 dt \right \} ^{\frac12}\cdot\|\psi\|_{W^{1,2}_0(\Omega)} \\[4mm]
 \leq & \left\{\displaystyle \int^\infty_{t_k}\|u_t(\cdot,s)\|^2_{(W^{1,2}_0(\Omega))^*}ds
\right \} ^{\frac12}\cdot\|\psi\|_{W^{1,2}_0(\Omega)}\\
\longrightarrow &0~~\hbox{as}~~k \rightarrow\infty,
  \end{array}
 \end{equation*}
due to Lemma \ref{lemma5.5}.

Let $\mu(t)=\left(\frac{1}{|\Omega|}\int_{\Omega}u^{\frac{m+1}{2}}(\cdot,t)\right)^{\frac{2}{m+1}}$. Then as in \eqref{5.18}, we have
\begin{equation*}
\overline{u_0}^{m-1}\int_{\Omega}|u(\cdot,t)-\mu(t)|^2\leq \int_{\Omega}|u^{\frac{m+1}{2}}(\cdot,t)-\mu^{\frac{m+1}{2}}(t)|^2\leq c_5\int_{\Omega}|\nabla u^{\frac{m+1}{2}}|^2	
\end{equation*}
and thus
\begin{equation}\label{5.29}
\overline{u_0}^{m-1}\displaystyle \int^{t_k+1}_{t_k}\int_{\Omega}|u(\cdot,t)-\mu(t)|^2\leq  c_5\displaystyle \int^{t_k+1}_{t_k}\int_{\Omega}|\nabla u^{\frac{m+1}{2}}(\cdot,t)|^2.	
\end{equation}
We now introduce $$  u_k(x, s) := u(x, t_k + s), (x, s)\in\Omega\times (0, 1) $$
and $$  \mu_k(x, s) := \mu(x, t_k + s), (x, s)\in\Omega\times (0, 1)  $$
for $k\in\mathbb{N}$.  Then \eqref{5.29} implies that
\begin{equation*}
\begin{array}{rl}
\overline{u_0}^{m-1}\displaystyle \int^{1}_{0}\int_{\Omega}|u_k(\cdot,s)-\mu_k(s)|^2ds\leq &
c_5\displaystyle \int^{t_k+1}_{t_k}\int_{\Omega}|\nabla u^{\frac{m+1}{2}}(\cdot,t)|^2\\
& \longrightarrow 0~~\hbox{as }~~k\rightarrow\infty,
\end{array}
\end{equation*}
due to \eqref{5.8} in Lemma \ref{lemma5.3}. This means that
\begin{equation}\label{5.30}
u_k(x,s)-\mu_k(s)\rightarrow 0~~\hbox{in} ~~L^2(\Omega\times (0,1))~~\hbox{as}~~k\rightarrow\infty,
\end{equation}
which in particular allows us to get
 \begin{equation}\label{5.31}
\displaystyle \int^{1}_{0}\int_{\Omega}(u_k(\cdot,s)-\mu_k(s)) \psi (\cdot) ds\rightarrow 0~ \hbox{as}~~k\rightarrow\infty
\end{equation}
as well as
\begin{equation}\label{5.32}
\displaystyle \int^{1}_{0}\int_{\Omega}(u_k(\cdot,s)-\mu_k(s)) ds\rightarrow 0~ \hbox{as}~~k\rightarrow\infty.
\end{equation}
thanks to the weak convergence of $L^2(\Omega\times (0,1))$.
  Moreover, by Lemma \ref{lemma5.4}, we have
\begin{equation*}
 \displaystyle \int^{t_k+1}_{t_k}\int_{\Omega}w(\cdot,t)dt\leq |\Omega|\|w(\cdot,t_k)\|_{L^\infty(\Omega)}
 \rightarrow 0~ \hbox{as}~~k\rightarrow\infty
 \end{equation*}
and thereby
 \begin{equation}\label{5.33}
\begin{array}{rl}
  |\Omega|\displaystyle\int^{1}_{0} \mu_k(s) ds=&\displaystyle \int^{1}_{0}\int_{\Omega}u_k(\cdot,s)ds-
\displaystyle \int^{1}_{0}\int_{\Omega}(u_k(\cdot,s)-\mu_k(s)) ds\\
=& |\Omega|u_*-\beta\displaystyle \int^{t_k+1}_{t_k}\int_{\Omega}w(\cdot,t)dt
-
\displaystyle \int^{1}_{0}\int_{\Omega}(u_k(\cdot,s)-\mu_k(s)) ds\\
\rightarrow & |\Omega|u_*
~ \hbox{as}~~k\rightarrow\infty
 \end{array}
 \end{equation}
due to \eqref{5.32} and \eqref{5.20}.

Therefore from  \eqref{5.28},   \eqref{5.31} and \eqref{5.33}, it follows that
\begin{equation*}
\begin{array}{rl}
\displaystyle\frac{c_1}2 \leq & \displaystyle\int^{t_k+1}_{t_k}\int_\Omega u(\cdot, t)\psi(\cdot) dt-\displaystyle \int^{t_k+1}_{t_k}\int_\Omega u_{\star}\psi (\cdot) dt
\\[4mm]
=& \displaystyle \int^{1}_{0}\int_\Omega (u_k(\cdot, s)-\mu_k(s)) \psi(\cdot)ds+\int^{1}_{0}\int_\Omega \mu_k(s) \psi(\cdot)ds
 -\displaystyle  u_{\star}\int_\Omega \psi(\cdot) \\[3mm]
= & \displaystyle \int^{1}_{0}\int_\Omega (u_k(\cdot, s)-\mu_k(s)) \psi(\cdot)ds+\int^{1}_{0} \mu_k(s) ds\int_\Omega \psi(\cdot)
 -u_{\star} \displaystyle  \int_\Omega \psi(\cdot) \\
\rightarrow &0~ \hbox{as}~~k\rightarrow\infty,
 \end{array}
\end{equation*}
which is absurd and hence proves that actually \eqref{5.27} is valid.

Let us suppose on the contrary that \eqref{5.26} be  false. Then without loss of generality  there exist sequence $\{x_k\}_{k\in \mathbb{N}}$ and $\{t_k\}_{k\in \mathbb{N}}\in (0,\infty)$
with $t_k\rightarrow \infty$ as $k\rightarrow \infty$ such that for some $c_1>0$
$$
u(x_k,t_k)-u_*=\max_{x\in \Omega}|u(x,t_k)-u_*|\geq c_1~~\hbox{for all} ~k\in  \mathbb{N}.
$$
In view of the compactness of $\overline\Omega$, where passing to subsequences we can find $x_0\in \overline\Omega$ such that  $x_k\rightarrow x_0$ as $k\rightarrow \infty$. Furthermore,
because $u$ is uniformly continuous in $\bigcup_{~k\in  \mathbb{N}}(\overline\Omega\times t_k)$,
 this
entails that one can extract a further subsequence if necessary such that
 $$
u(x,t_k)-u_*\geq  \frac{c_1}2 ~~~\hbox{for all}~ x\in B:=B_{\delta}(x_0)\cap \Omega~ \hbox{and}~k\in \mathbb{N}
$$
for some $\delta>0$. Noticing that if $ x_0\in \partial\Omega $, the smoothness of $\partial\Omega$  ensures the existence of $\hat{x}_0\in \Omega$
and a smaller $\hat{\delta}>0$ such that $B_{\hat\delta}(\hat x_0)\subset B$.
Now  taking the nonnegative function $\psi\in C^\infty_0(B_{\hat\delta}(\hat x_0)))$  such as a smooth truncated function in $B_{\hat\delta}(\hat x_0))$, we then have
$$
\int_\Omega (u(x, t_k)-u_{\star} )\psi dx=\int_{B_{\hat\delta}(\hat x_0)}(u(x, t_k)-u_{\star} )\psi dx\geq \frac{c_1}2\cdot \int_\Omega \psi dx,
$$
which contradicts \eqref{5.27}  and hence proves the lemma.
\end{proof}

\subsection{Stabilization of $v$}
In what follows, based on  the uniform H\"{o}lder bounds of $v$ and decay of $\nabla v $ implied by \eqref{5.2} and \eqref{5.9} respectively,
we shall show the   corresponding stabilization result for $v$ by a contradiction argument.
\begin{lemma}\label{lemma5.7}
Let $m>1$ and $(u,v,w)$ be the solution  of \eqref{1.10}--\eqref{1.12} obtained in Lemma \ref{Lemma2.1}. Then we have
\begin{equation}\label{5.34}
\|v(\cdot,t)-u_{\star}\|_{L^{\infty}(\Omega)}\rightarrow 0\quad \textnormal{\emph{as $t\rightarrow \infty$}}.	
\end{equation}	
\end{lemma}
\begin{proof} 
 Combined  the uniform H\"{o}lder bounds of $v$ and decay of $\nabla v $ implied by \eqref{5.2} and \eqref{5.9}, respectively,
\eqref{5.34} may be  derived by  a contradiction argument. Indeed, assume that  (\ref{5.34}) was false, then we can find a sequence $(t_k)_{k\in\mathbb{N}}$ with $t_k\rightarrow \infty$ as $k\rightarrow \infty$,  and constant $c_1>0$ such that
\begin{equation*}
\|v(\cdot,t_{k})-u_{\star}\|_{L^{\infty}}\geq c_1.	
\end{equation*}
Furthermore the  uniform H\"{o}lder continuity of $v$ in $\Omega\times[t,t+1]$  warrants  the existence of  $(x_k)_{k\in\mathbb{N}}$ and $r>0$ such that
\begin{equation*}
|v(x,t)-u_{\star}|>\frac{c_1}{2}	
\end{equation*}
for every $x\in B_{r}(x_k)$	and $t\in(t_k,t_k+\tau)$ and hence
\begin{equation}\label{5.37}
\int^{t_k+\tau}_{t_{k}}\int_{\Omega}|v(\cdot,t)-u_{\star}|^2>\frac{|\Omega|\tau c^2_1}{4}.	
\end{equation}

On the other hand, the Poincar\'e inequality indicates
\begin{equation}\label{5.38}
\int^{t_k+\tau}_{t_k}\int_{\Omega}|v(\cdot,t)-u_{\star}|^2\leq C\int^{t_k+\tau}_{t_k}\int_{\Omega}|\nabla v|^2+C
\int^{t_k+\tau}_{t_k}\int_{\Omega}|\overline{v(\cdot,t)}-u_{\star}|^2.
\end{equation}
Therefore \eqref{5.38} yields a contradiction to \eqref{5.37} thanks to
$$
\int^{t_k+\tau}_{t_k}\int_{\Omega}|\nabla v|^2\rightarrow 0 \quad\textnormal{as $t_k\rightarrow \infty$}.	
$$ \end{proof}

Now the convergence result in the flavor of Theorem 1.1 has actually been proved
already.

\textbf{Proof of Theorem 1.1} The claimed assertion in Theorem 1.1 is the consequence of  Lemma \ref{lemma5.4},  \ref{lemma5.6} and \ref{lemma5.7}.
\section{Acknowledgments}
 This work is supported by the NNSF of
China (No.12071030) and Beijing key laboratory on MCAACI.

\end{document}